\documentclass[a4paper,10pt]{article}

\usepackage[english]{babel}
\usepackage[T1]{fontenc}
\usepackage{enumerate}
\usepackage{amsmath}
\usepackage{amsthm}
\usepackage{amssymb}

\newcommand{\Qp}{\mathbb{Q}_p}

\newcommand{\Zp}{\mathbb{Z}_p}

\newcommand{\Fp}{\mathbb{F}_p}

\newcommand{\Q}{\mathbb{Q}}
\newcommand{\R}{\mathbb{R}}

\newcommand{\C}{\mathbb{C}}
\newcommand{\Z}{\mathbb{Z}}
\newcommand{\N}{\mathbb{N}}

\newtheorem{theorem}{Theorem}[section]

\newtheorem{prop}[theorem]{Proposition}
\newtheorem{lemma}[theorem]{Lemma}
\newtheorem{cor}[theorem]{Corollary}
\newtheorem{definition}[theorem]{Definition}
\newtheorem{Fact}{Fact}
\newtheorem{claim}[theorem]{Claim}

\theoremstyle{remark}
\newenvironment{Remark}{\begin{trivlist}\item[\hskip \labelsep {\bfseries Remark.}]}{\end{trivlist}}

\newenvironment{notation}{\begin{trivlist}\item[\hskip \labelsep {\bfseries Notation.}]}{\end{trivlist}}

\newcommand{\Square}[2]{\,\square_{#1}^{#2}\,}

\title{Model theory of the field of $p$-adic numbers expanded by a multiplicative subgroup.}
\author{Nathana\"el Mariaule\footnote{During the preparation of this paper the author was partially supported by ANR-13-BS01-0006 and by the Fonds de la Recherche Scientifique - FNRS}} 
\date{}
\begin{document}
\maketitle
\begin{abstract}
Let $G$ be a multiplicative subgroup of $\Qp$. In this paper, we describe the theory of the pair $(\Qp, G)$ under the condition that $G$ satisfies Mann property and is small as subset of a first-order structure. First, we give an axiomatisation of the first-order theory of this structure. This includes an axiomatisation of the theory of the group $G$ as valued group (with the valuation induced on $G$ by the $p$-adic valuation). If the subgroups $G^{[n]}$ of $G$ have finite index for all $n$, we describe the definable sets in this theory and prove that it is NIP. Finally, we extend some of our results to the subanalytic setting.  
\end{abstract}

\par Let $\mathcal{M}=(M,\cdots)$ be a $\mathcal{L}$-structure and $A$ be a subset of $M$. In various context, people have studied the theory of the pair $(M,A)$ (in the language $\mathcal{L}$ expanded by a unary predicate interpreted by $A$). This problem has been discussed in a pure abstract setting e.g.\cite{C-S} or in particular cases of $M$ and $A$. Well-known examples are the pairs of fields: $M$ is a field and $A$ is an elementary substructure e.g.\cite{Robinson}. An other example that has been studied by many authors is when $K$ is a field and $A$ is a multiplicative subgroup of $K^*$. The first instances of this problem are L. van den Dries \cite{vdDries5} where he works on the theory of the pair $(\R, 2^\Z)$ or B. Zilber \cite{Zilber} where the pair $(\C, \mathbb{U})$ is considered ($\mathbb{U}$ is the group of roots of unity). Both these results have been generalised to the case where $K$ is a real or algebraically closed field in \cite{Gunaydin-vdD} and $A$ is a small multiplicative subgroup.

\par In this paper, we consider the same problem where $K$ is now the field of $p$-adic numbers $\Qp$. A special case of this problem is presented in \cite{Mariaule} ($G$ is $n^{\Z}$ with $n\in \N$). The aim of this paper is first to generalise the results proved there (working with a general subgroup $G$) and second to discuss classical problems that were not considered previously (extension to subanalytic setting and NIPness of the theory). In section \ref{Expansion of Qp by a multiplicative subgroup}, we axiomatise the theory of the pair $(\Qp, G)$. We will see in this section that $G$ can be reduced  to two main parts (definably in a suitable language): a discrete cyclic group and a subgroup $D$ that is dense in an open subgroup $1+p^n\Zp$. On the discrete group, the $p$-adic valuation induces a structure of ordered group. This part is a $p$-adic equivalent to \cite{vdDries5}. If we add a function symbol $\lambda$ interpreted by a function that sends an element of the field to an element of the group with same valuation, we get quantifier elimination. On the dense part $D$, the valuation induces a structure of valued group. The theory of this valued group is part of the theory of the pair $(\Qp,G)$. In section \ref{valued groups}, we will study the model theory of $D$ as valued group. We will axiomatise the theory of this group, prove that its theory admits quantifier elimination in a natural language and that it is NIP.

\par The axiomatisation of these two cases is part of the axiomatisation of the theory $(\Qp,G)$. Two extra properties are required in order to obtain a complete theory. First, we will assume that $G$ has Mann property (this gives us control on the field $\Q(G)$). This property was used in \cite{Zilber}\cite{Gunaydin-vdD} and it is quite natural to use it in the same way in our context. Finally we will assume that $G$ is small in some sense (this is the case if $G$ is finitely generated for instance). In section \ref{definable sets}, we describe the definable sets in models of our theory assuming that $G^{[n]}:=\{g^n:g\in G\}$ has finite index in $G$ for all $n$. We also describe the subsets of $G$ that are definable in $(\Qp,G)$ in the case where the discrete part of $G$ is trivial.

\par In section \ref{Expansion for subanalytic structure}, we consider the same structure in the language of subanalytic sets $\mathcal{L}_{an}^D$ in the sense of \cite{Denef-vdD}. In this language, the dense part of the group might define the rings of integers. So we assume that $G$ is a discrete group (e.g. $G=p^\Z$). In that case, we prove that the theory of $(\Qp, G)$ admits the elimination of quantifiers in the language $\mathcal{L}_{an}^D(\lambda, A)$ where $\lambda$ is as defined before.
\par We apply this result to the study of the $p$-adic Iwasawa logarithm. Let us recall that the logarithm defined by the usual power series $\log_p(1+x)=\sum_{n>0} (-1)^{n+1}x^n$ is convergent in $\Qp$ only for the elements with positive valuation. On the other hand, as $\Qp^*=p^\Z\times \mu_{p-1}\times (1+p\Zp)$ (where $\mu_{p-1}$ is the group of roots of unity of order $p-1$), one can define a morphism of groups $LOG:(\Qp^*,\cdot)\rightarrow (\Qp,+)$ that extends $\log_p(1+x)$. Furthermore, this map is unique up to the choice of the value of $p$. Set $LOG(p)=0$. Then for all $x\in \Qp$; $x=p^n\xi y\in \Qp$ for some unique $n\in \Z$, $\xi\in \mu_{p-1}$ and $y\in 1+p\Zp$. So, because $LOG$ is a morphism of groups,
$$LOG(x)=nLOG(p)+LOG(\xi)+LOG(y)=\log_p(y).$$
 This logarithm map is called the $p$-adic Iwasawa logarithm. The expansion $(\R,exp)$ of $\R$ has been extensively studied by various authors. In the $p$-adic context, the usual exponential is convergent only on $p\Zp$. There is no natural structure of exponential field on $\Qp$. On the other hand the $p$-adic Iwasawa logarithm induces a structure of logarithmic field. This structure is canonical up the the choice of $LOG(p)$ if we require the logarithm to be analytic on $1+p\Zp$.
In \cite{Denef-vdD} and many papers afterwards, the theory of $\Qp$ with restricted analytic structure has been studied. On the other hand, there is no proper expansion of $\Qp$ by a global function that has been studied in the context of model theory. The expansion by the Iwasawa Logarithm is an example of such a structure.
 As an application of our result on ($\Qp, G)$, we obtain that the theory of $(\Qp, LOG)$ admits the elimination of quantifier in the language $\mathcal{L}_{an}^D(\lambda, A, \xi)$. In section \ref{Iwasawa Logarithm}, we prove that in some (countable) reduct, this theory is model-complete using techniques from \cite{Mariaule2}.

\par In the last section \ref{NIP}, we prove that the following theories are NIP: (1) $Th(\Qp, p^{\Z})$ in the language $\mathcal{L}_{an}^D(G,\lambda)$ and (2) $Th(\Qp, G)$ with $G$ a dense subgroup of $1+p^n\Zp$. This is an application of the abstract setting in \cite{C-S} where it is proved that the theory of $Th(\mathcal{M},A)$ is NIP whenever it admits elimination of quantifiers up to bounded formulas, $Th(\mathcal{M})$ is NIP and the theory of the structure induced on $A$ is NIP. In our case, the first and last points are essentially proved in section \ref{definable sets} and the second hypothesis is well-known. As consequence of (1) , $Th(\Qp, LOG)$ is NIP. 

\begin{notation} Let $K$ be a valued, we will denote by $v_K$ or $v$ its valuation, by $\mathcal{O}_K$ its valuation ring, by $vK$ its value group, by $res\ K$ its residue field and by $res$ the residue map. The $p$-adic valuation will be denoted $v_p$. Let $K^h$ denote the henselisation of $K$. Let $A$ be a ring, we denote by $A^*$ the set of nonzero elements and by $A^{\times}$ the set of units. We will denote by $\mathcal{L}_{Mac}$ the language of $p$-adically closed fields $(+,-,\cdot,0,1, P_n (n\in\N))$ where $P_n$ is interpreted in $\Qp$ by the set of $n$th powers. If $G$ is a group, $G_{Tor}$ denotes its torsion part.
\end{notation}

\section{Elementary properties of abelian $p$-valued groups}\label{valued groups}
\par Let $G$ be a subgroup of $(\Zp,+,0)$. Then, $v_p$ induces a map on $G$ such that $(G,v_p)$ is a valued group. In this section, we study the theory of the structure $(G,v_p)$. In a first time, we axiomatise the theory of this group and prove a result of quantifier elimination. At the end of this section, we will also prove that $Th(G,v_p)$ is NIP if the index $[G:nG]$ is finite for all $n$. The special case $G=\Z$ has been discussed in \cite{Guignot-thesis}\cite{Mariaule}.
\begin{definition}
Let $(G,+,0_G)$ be an abelian group and $V:G\twoheadrightarrow \Gamma\cup\{\infty\}$ where $\Gamma$ is a totally ordered set with discrete order and no largest element and $\infty$ is an element such that $\infty>\gamma$ for all $\gamma\in \Gamma$. We say that $(G,V)$ is a \emph{$p$-valued group} if for all $x,y\in G$ and for all $n\in \Z$,
\begin{itemize}
	\item $V(x)=\infty$ iff $x=0_G$;
	\item $V(nx)=V(x)+v_p(n)$;
	\item $V(x+y)\geq \min\{V(x),V(y)\}$;
\end{itemize}
where $v_p$ is the $p$-adic valuation, $nx=x+\cdots +x$ ($n$ times), $(-n)x=-(nx)$ for all $n>0$, $0x=0_G$ and if $x\in G$, $V(x)+k$ denotes the $k$th successor of $V(x)$ in $\Gamma\cup\{\infty\}$ (by convention the successor of $\infty$ is $\infty$). 
\end{definition}
\begin{Remark} The above axioms implies that if $V(x)\not=V(y)$ then $V(x+y)=\min\{V(x), V(y)\}$.
\end{Remark}
Let $G$ be an abelian group. We denote by $[n]G$ the index of $nG$ in $G$. If it is not finite, we set $[n]G=\infty$ with no distinction between cardinalities. Let $\mathcal{L}_{pV}$ be the two-sorted language $\Big((+,-,0, 1,\equiv_n (n\in\N),c_{ij}),(<,S,0_{\Gamma},\infty),V\Big)$ where the first sort corresponds to the group $G$ and the second to the ordered set $V(G)$. $c_{ij}$ is a collection of constant symbols interpreted by representatives for the cosets of $iG$. We fix $(q_n,n\in \N$) an enumeration of the prime numbers. Let $N=(t_1,t_2,\cdots)$ be a sequence in $\N\cup\{\infty\}$ where the index $t_n$ associated to $p$ is nonzero (we will assume $n=1$).  By convention, we set $q_n^\infty:=\infty$. Let $T_{pV,N}$ be the $\mathcal{L}_{pV}$-theory axiomatised as follows: Let $\Big((G,+,-,0,1,\equiv_n (n\in\N), c_{ij} (0\leq j\leq q_i^{t_i})),(VG\cup\{\infty\},<,S,0_{\Gamma},\infty),V\Big)$ be a model of $T_{pV,N}$ then
\begin{itemize}
	\item $(G,+,-,0)$ is an abelian group, $x\equiv_n y$ iff $\exists g\in G, x=y+ng$ and $[q_n]G=q_n^{t_n}$, $c_{ij}\not\equiv_i c_{ik}$ for all $i$ and for all $k\not=j$, $c_{i0}=0$;
	\item $(VG,<)$ is a discrete ordered set with first element $0_{\Gamma}$, any nonzero element has a predecessor and there is no last element, $\infty$ is an element such that $\gamma<\infty$ for all $\gamma\in VG$ and $S$ is the successor function ($S(\infty):=\infty$);
	\item $(G,V)$ is a $p$-valued group;
	\item $1$ is an element such that $V(1)=0_{\Gamma}$ and for all $n$ such that $[n]G\not=1$, $0, 1,\cdots, n-1:=1+\cdots +1$ are in distinct classes of $nG$, $c_{nk}=k$ for all $k\leq n-1$;
	\item For all $x,y\in G$, if $V(x)=V(y)$, then there is a unique $0<i<p$ such that $V(x-iy)>V(x)$;
	\item $G$ is regularly dense i.e. for all $n$, $nG$ is dense in $\{x\in G\mid\ V(x)\geq v_p(n)\}$ (where $v_p(n)$ denotes the $v_p(n)$th successor of $0_{\Gamma}$ in $VG$) i.e. for all $n$
	$$\forall x\in G\ V(x)\geq v_p(n)\rightarrow\Big[\forall\gamma\geq v_p(n)\in VG\exists y\in nG\ V(x-y)\geq \gamma\Big];$$
	\item $G$ is equidistributed i.e. for all $n$ such that $[n]G$ is infinite, for all $\gamma\in VG$, for all $c\in G$, the set $\{x\mid\ V(x-c)\geq \gamma\}$ contains representatives for infinitely many classes of $[n]G$.
\end{itemize}
\begin{Remark}
	1. If $[n]G=k$, then it is written as 
	$$\Psi_{n,k}\equiv \exists x_1,\cdots, x_k\in G \bigwedge_{i\not=j}x_i\not\equiv_n x_j\wedge \forall y\in G, \bigvee_{i=1}^k y\equiv_nx_i.$$
	If $[n]G=\infty$, then it is written as the scheme of formulas $\Phi_{n,k}$ (for all $k\in \N$) where
	$$\Phi_{n,k}\equiv \exists x_1,\cdots, x_k\in G \bigwedge_{i\not=j}x_i\not\equiv_n x_j.$$
	Similarly, equidistribution is described by the collection of the following axioms: for all $n$ such that $[n]G=\infty$ and for all $m$
	$$\varphi_{n,m}\equiv \forall \gamma\in VG\ \forall c\in G\ \exists x_1,\cdots, x_m\ \bigwedge_i V(x_i-c)\geq \gamma\bigwedge_{i\not= j} x_i\not\equiv_n x_j.$$
	\par 2. $t_n=1$ for all $n$, then $T_{pV,N}$ is the theory of $p$-valued $\Z$-group that was studied in \cite{Mariaule} and \cite{Guignot-thesis}.
	\par 3. The last axiom is true for subgroups of $\Zp$ by the pigeon-hole principle and by compactness of $\Zp$. Note that if $G$ is a subgroup of $\Zp$ and $t_n$ is the value such that $[G:q_nG]=q_n^{t_n}$ then $(G,v_p)\vDash T_{pV,N}$.
	\par 4. If $M'$ is a $p$-valued group, then it is torsion-free. In particular, any model of $T_{pV,N}$ is torsion-free.
\end{Remark}
\begin{theorem}\label{EQ valued groups}
The theory $T_{pV,N}$ is consistent and admits the elimination of quantifiers.
\end{theorem}
\par First, let us prove the consistency of the theory. By the above remark it is sufficient to find a subgroup $\Zp$ with the right indices. Fix $q$ a prime number, say $q$ is the $n$th prime number in our list. We assume that $t_n\geq 1$. Let $x_1,\cdots, x_{t_n-1}$ be elements of $\Zp$ algebraically independent over $\Q$. 
Let $G_1=\Z \oplus_i x_i\Z(p,q)\subset \Zp$ with 
$$\Z(p,q)=\{a/b\mid a,b\in \Z \mbox{ and $p$,$q$ does not divide } b\}.$$
Then, for all $q'$ prime not equal to $p,q$, $[G_1:q'G_1]=q'$. Indeed, any element of $G_1$ can be written as $x=n+\sum (a_i/b_i)x_i = n+ q'\sum (a_i/b_iq')x_i$ where $n\in \Z$ $n_i\in \Z(p,q)$ i.e. $x\in \Z+q'\Z(p,q)$. So, $0,\cdots, q'-1$ form a set of representatives of the cosets of $q'G_1$ (by algebraic independance of the $x_i$'s). On the other hand, $[G_1:pG_1]=p^{t_n}$ and $[G_1:qG_1]=q^{t_n}$.  Indeed, the collection $\{n+\sum n_ix_i:\  0\leq n,n_i<q, n,n_i\in \Z\}$ forms a set of representatives of the cosets of $qG_1$. Similarly for the prime number $p$.

\par In order to reduce $[p]G_1$, we define $y_{ij}:= \frac{1}{p^j}\sum_{k\geq j}x_{ik}p^k\in \Zp$ where $x_{i}=\sum_{k\geq 0} x_{ik}p^k$ with $x_{ij}\in \{0,\cdots, p-1\}$. Note that $y_{i0}=x_i$. Let 
\begin{align*}
G_2:&=G_1\bigoplus_{j\in \N_0} y_{ij}\Z(p,q),\\
		&=\{n+\sum_{0\leq i,j<R} n_{ij}y_{ij}\mid\ n\in \Z, n_{ij}\in \Z(p,q)\}.
\end{align*}
Then for all $q'$ prime not equal to $q$, $[G_2:q'G_2]=q'$. Indeed, for $q'$ not equal to $p$, we argue like above. For $q'=p$, we notice that $y_{ij}=x_{ij}+py_{i(j+1)}\in \Z+pG_2$. So, $0,\cdots, p-1$ form a set of representatives of the cosets of $pG_2$. 
\par Finally, $[G_2:qG_2]=q^{t_n}$: Let us show that $\{n_0+\sum n_ix_i:\ 0\leq n_k<q\}\subset G_1$ forms a set of representatives of $qG_2$. For it is sufficient to prove that $n_0+\sum n_ix_i\notin qG_2$ (if nonzero) and that for all $i,j$, $j>0$, $y_{ij}\in G_1+qG_2$. The first part is true as $n_0+\sum_{i\leq t_n-1} n_ix_i\in qG_2$ iff 
$$n_0+\sum_{i\leq t_n-1} n_ix_i= qn'+\sum_{i,j \leq s} qn'_{ij}y_{ij}$$
 for some $n'\in \Z, n'_{ij}\in \Z(p,q)$ and $s\in \N$. The right hand side can be rewritten as $q\widetilde{n}_0+\sum_i q\widetilde{n}_ix_i$ for some $\widetilde{n}_k\in\Z(q)$ as $y_{ij}=(x_i-\sum_{k<j}x_{ik}p^k)/p^j\in x_i\Z(q)+\Z(q)$. As the $x_i$'s are $\Q$-algebraically independent, $n_0+\sum n_ix_i=q\widetilde{n}_0+\sum_i q\widetilde{n}_ix_i$ iff $n_k=q\widetilde{n}_k$ for all $k$ iff $n_k= 0$ for all $k$ (as $0\leq n_k<q$). For the last claim, let us remark that $p^{j}y_{ij}=r+x_i$ for some $r\in \Z$. Let $a,b\in \Z$ such that $aq+bp^j=1$. Then, $bp^jy_{ij}=br+bx_i$. So, $y_{ij}=(aq+bp^j)y_{ij}=br+bx_i+qay_{ij}\in G_1+qG_2$.

\par Write $G_2$ as $\Z\oplus G_2'$. We define similarly ${G'}_{2,q'}$ for all $q'=q_n$ prime number such that $t_n$ is nonzero (we choose the $x_{i,q}$'s all distinct among a transcendence basis of $\Zp$ over $\Q$; that is an uncountable set). Note that for $q'=p$, it is actually sufficient to take ${G'}_{2,p}:=G_{1,p}$. Then, 
$$G:=\Z(q_i: i \in I) \bigoplus_{\{q_i: i\notin I\}} {G'}_{2,q_i}$$
 is a model of our theory where $I=\{i:\ t_i=0\}$ (where the valuation on $G$ is induced by the $p$-adic valuation). The axioms of indices are satisfied by construction: by definition, $qG_{2,q'}'=G_{2,q'}'$ if $q\not=q'$ and as above $\{n_0+\sum_{k} n_{k}x_{k,q_n}\mid\ 0\leq n_0, n_{k}<q_n\}$ forms a set of representatives of the cosets of $q_nG$. The other axioms are true as $G$ is a subgroup of $\Zp$.

\begin{definition}\label{pure subgroup} Let $M'$ be a model of $T_{pV,N}$ and $M$ be a pure subgroup of $M'$. Let $c\in M'\setminus M$. We define
$$M\langle c\rangle = \{x\in M'\mid nx=mc+y\mbox{ where }y\in M, m,n\in \Z\}.$$
\end{definition}
It is not hard to see that $M\langle c\rangle$ is a pure subgroup of $M'$.

Let $\Square{}{}_i$ denote either $<,>$ or $=$.
\begin{lemma}\label{balls system}Let $M, M'$ be models of $T_{pV,N}$ and $M_0$ be a pure common subgroup of $M$ and $M'$. Let $\overline{a}\subset M_0$,$\overline{n},\overline{m}\subset \Z$ and $\overline{\gamma}\subset VM_0\cup \{\infty\}$. Let
$$\alpha(X)\equiv \bigwedge_i
\left\{
\begin{array} {l} 
V(X-a_i)+n_i\Square{i}{} V(X-a_j)+m_j\\
V(X-a_i)+n'_i\Square{i}{} \gamma_i.
\end{array}
\right. 
$$
Then, $M\vDash \exists x \alpha(x)$ iff $M'\vDash \exists x\alpha(x)$.
\end{lemma}
\begin{proof}

Let $x\in M$ be a solution of $\alpha(X)$. Note that if the formula $V(X-a_i)\geq \infty$ occurs in $\alpha$ then $x=a_i$ so we are done. Furthermore, $V(X-a_i)\leq\infty$ is always true. So, we may assume that $\gamma_i\not=\infty$ for all $i$.
If $x\in M_0$, we are done. Otherwise, we consider different cases:
\par (a) First, if $V(x-a_i)\in VM_0$ for all $i$, then we find that $\alpha(X)$ is realised in $M_0$. For we may assume that $V(x-a_1)$ is maximal among the $V(x-a_i)$'s. If $V(x)>V(a_1)$, $V(x-a_i)=V(a_i)$ for all $i$. So any $y\in M_0$ such that $V(y)>V(a_1)$ is solution of $\alpha(X)$. Otherwise, by the axiom of $p$-valued groups, there is $0<t<p$ such that $V(x-ta_1)>V(x-a_1)$. Take $y=ta_1$. Then for all $i$, $V(y-a_i)=\min\{V(x-y), V(x-a_i)\}=V(x-a_i)$ i.e. $y$ is solution of $\alpha(X)$.

\par (b) There is $i$ such that $V(x-a_i)\notin VM_0$. Let $j\not=i$. Then, $V(x-a_j)=\min\{V(x-a_i), V(a_i-a_j)\}$ as $V(x-a_i)\not=V(a_i-a_j)\in VM_0$. Let $\widetilde{\alpha}(X)$ be the system where we substitute in $\alpha$ $V(X-a_j)$ by $V(X-a_i)$ if $V(x-a_i)< V(a_i-a_j)$ and by $V(a_i-a_j)$ if $V(x-a_i)>V(a_i-a_j)$ and we add the inequalities $V(x-a_i)< V(a_i-a_j)$, $V(x-a_i)< V(a_i-a_j)$ according to the case where $j$ falls. Then, if $y\in M'$ is such that $M'\vDash \widetilde{\alpha}(y)$ then $V(y-a_j)=\min\{V(a_i-a_j),V(y-a_i)\}$. So, $M'\vDash \alpha(y)$.
\par Let $t:= V(x-a_i)$. Let $\widetilde{\alpha}_B(t)$ be the system of quantifier-free formulas in the language of discrete ordered sets obtained when we replace $V(X-a_i)$ by $t$. Let us recall that in our language, $Th(VM)=Th(VM')$ admits the elimination of quantifiers. So, there is $t'\in VM'$ such that $M'\vDash \widetilde{\alpha}(t')$. Let $y\in M'$ such that $V(y-a_i)=t'$. Then, $M'\vDash \widetilde{\alpha}(y)$.
\end{proof}

\begin{lemma}\label{congruences system} Let $M'$ be a model of $T_{pV,N}$ and $M$ be a structure that is a pure subgroup of $M'$. Let $\overline{k},\overline{l}\subset \Z$, $\overline{n},\overline{m}\subset \N$ and $\overline{a},\overline{b}\subset M$. Let 
$$\beta(X)\equiv \bigwedge_i k_iX\equiv_{n_i}a_i\ \wedge\bigwedge_j \neg l_jX\equiv_{m_j}b_j.$$
If $M'\vDash \exists X \beta(X)$ then $M\vDash \exists X \beta(X)$. Furthermore, if $c\in M$ is so that $M\vDash \beta(c)$, then there is $t\in \N$ such that for all $y\in M$ with $y\equiv_t c$, $M\vDash \beta(y)$.
\end{lemma}
\begin{Remark} For all $n$, $[n]M'=[n]M$. If $[n]M'<\infty$, this follows from the pureness assumption. If $[n]M'=\infty$, $[n]M=\infty$ as $M$ contains the constants $c_{nk}$ and by definition of our theory there are infinitely many $k\not=l$ such that $c_{nk}\not\equiv_nc_{nl}$.
\end{Remark}
\begin{proof}
Let $K=\prod k_il_j$. Then, $k_iX\equiv_{n_i}a_i$ is equivalent to $KX\equiv_{n_i(K/k_i)} (K/k_i)a_i$; similarly for the incongruences. We proceed to a change of variables $Y=KX$ and we extend $\beta$ by $Y\equiv_K 0$. So, up to this transformation,  we may assume $k_i=l_i=1$. Next, if there is no incongruences, then by the following fact we are done.
\begin{Fact}
$\bigwedge_i X\equiv_{n_i}a_i$ has a solution in $M'$ iff $gcd(n_i,n_j)$ divides $(a_i-a_j)$. In this case the set of solution is of the form $X\equiv_{n} a$ for some $a\in M$ and $n=lcm(n_1,\cdots, n_k)$.
\end{Fact} 

This is a classical result for system of congruences in $\Z$. Its proof can be easily adapted in our theory.  So, $\beta$ is equivalent to 
$$X\equiv_{n}a\ \wedge\bigwedge_j \neg X\equiv_{m_j}b_j $$
for some $n\in \N$ and $a\in M$. Next, we may assume that $[m_j]M=\infty$. For assume $[m_1]M$ finite. Then let $y$ be a solution of $\beta$ in $M'$. As $[m_1]M$ finite and $M$ pure in $M'$, there is $c\in M$ such that $c\equiv_{m_1} y$. It is sufficient to prove that the system $X\equiv_n a\wedge X\equiv_{m_1} c\wedge \bigwedge_{j>1} \neg X\equiv_{m_j}b_j$ has a solution in $M$. We can repeat the argument for each $j$ such that $[m_j]M$ is finite then apply again the above fact.

\par We deal now with incongruences of subgroups with infinite index. Fix $j$. By the last fact, $X\equiv_{n}a \wedge X \equiv_{m_j}b_j$ is equivalent to $X\equiv_{lcm(m_j,n)}{b'}_j$ for some ${b'}_j\in M$. So, $X\equiv_{n}a \wedge \neg X\equiv_{m_j}b_j$ is equivalent to $X\equiv_{n}a \wedge \neg X\equiv_{lcm(m_j,n)}{b'}_j$. So, we may assume that our system is of the type
$$X\equiv_{n}a\ \wedge\bigwedge_j \neg X\equiv_{m_j}b_j \qquad (\beta')$$
where $n$ divides $m_j$ and $[nM:m_jM]=\infty$.
\par The lemma is now a corollary of \cite{Neumann} Lemma 4.1: 
\begin{Fact}\label{Neumann Lemma}[Neumann \cite{Neumann} Lemma 4.1] Let $G$ be a group and $C_1,\cdots, C_n$ be subgroups of $G$. Then if for some $g_i\in G$, $G=\bigcup_i C_ig_i$, then at least one of the $C_i$ has finite index.
\end{Fact}
As $[nM:m_jM]=\infty$, by the above fact, $(\beta')$ has a solution $c$ in $M$. The last statement is immediate: Let $t$ be the least common multiple of $n,m_1,\cdots, m_k$ then any $x\in M$ such that $x\equiv_t c$ is solution of $(\beta)$.
\end{proof}

\begin{prop}\label{extension isom}
 Let $M',N'$ be models of $T_{pV,N}$ and $M,N$ be pure subgroups of $M'$ and $N'$. Assume that $N'$ is $|M|$-saturated. Let $h:M\rightarrow N$ be an isomorphism of $p$-valued groups. Let $a\in M'\setminus M$. Then, there is $b\in N'$ such that for all $n\in \N$, $k,l,r\in\Z$, $x\in M$ and $\gamma\in VM$,
$$V(la-x)+r\Square{}{}\gamma\mbox{ iff } V(lb-h(x))+r\Square{}{}h(\gamma).$$
$$(\neg) ka-x\equiv_n 0 \mbox{ iff } (\neg)kb-h(x)\equiv_n 0 $$
Furthermore, $h$ can be extended to an isomorphism of $p$-valued groups $\widetilde{h}:M\langle a\rangle\rightarrow N\langle b\rangle$ such that $\widetilde{h}(a)=b$.
\end{prop}
\begin{proof}
Let $p(X)$ be the partial type given by formulas
$$V(lX-h(x))+r\Square{}{}h(\gamma)$$
$$(\neg)kX-h(x)\equiv_n 0 $$
where $M'\vDash (\neg)ka-x\equiv_n 0 $ and $M'\vDash V(a-x)\Square{}{}\gamma$ for all $x\in M$ and $\gamma\in VM$. It is sufficient by saturation to prove that the type is finitely consistent. Let $\Phi(X)$ be a conjunction of formulas in $p(X)$. Let $\alpha(X)$ and $\beta(X)$ be the subformula of $\Phi$ such that $\alpha$ involves only valuational (in)equalities, $\beta$ only involves (in)congruences and $\Phi\equiv \alpha\wedge \beta$. Let us remark that $p(X)$ does not contain a formula of the type $V(lX-h(y))=\infty$:  otherwise, $la-x=0$ in $M'$, so as $M$ is a pure, torsion-free subgroup of $M'$, $a\in M$; contradiction. Next, we may assume that $k=l=1$. Indeed, if we multiply each equation by a suitable integer, we may assume $k=l$. Then replace in each equation $kX$ by $X$ and add a congruence $X\equiv_k 0$ in the formula $\Phi$.
\par As in the proof of Lemma \ref{congruences system}, we may assume that $(\beta)$ is:
$$X\equiv_{n}c\ \wedge\bigwedge_j \neg X\equiv_{m_j}b_j$$
where $c,b_j\in N$, $n$ divides $m_j$ and $[nN':m_jN']=\infty$. By regular density, $c+nN'$ is a dense subset of $B:=B(c,v_p(n))$ (the ball of centre $c$ and valuative radius $v_p(n)$). So any solution of $\Phi(X)$ is in $B$. Let $\alpha'(X)\equiv \alpha(X)\wedge X\in B$. Let $\alpha'_M$ be the system obtained from $\alpha'$ when we apply $h^{-1}$ to all parameters. This system has a solution in $M'$ therefore $\alpha'$ has a solution $y$ in $N'$ by Lemma \ref{balls system}. Furthermore, there is $\gamma\in VN'$ such that $B'=B(y,\gamma)$ is a set of solutions of $\alpha'$ in $N'$ (take $\gamma$ be larger than any valuation that appears previously). Let us remark that $B'$ is a subset of $B$. By equidistribution, there is a solution $y'$ of $\beta$ in $B'$. For let $B''=B(0,v_p(n))\supset B(y-c,\gamma)$. Apply Fact \ref{Neumann Lemma} with $G=nN'\cap B''$ (this is a group) and $G'=m_iN'\cap B''$. By equidistribution, $[G:G']=\infty$. Therefore, by Fact \ref{Neumann Lemma} and equidistribution there is $y'\in B(y-c,\gamma)$, $y'\equiv_n 0$ and $y'\not\equiv_{m_j} b_j-c$ for all $j$.  
 So $y'+c$ is a realisation of $\alpha\wedge \beta$ in $N'$. 
 This proves that the type $p(X)$ is finitely consistent in $N'$. 
\par Let $b$ be a realisation of this type in $N'$. Then $\widetilde{h}$ is an isomorphism. Indeed it is immediate that this is a isomophism of groups. It remains to prove that it is a morphism of valued groups. For it is sufficient to show that for all $l,l'\in \N$, $r,r'\in \Z$, $x,x'\in M$,
$$V(la-x)+r\Square{}{}V(l'a-x')+r'\mbox{ iff } V(lb-h(x))+r\Square{}{}V(l'a-h(x'))+r'.$$
We replace $V(la-x)+r\Square{}{}V(l'a-x')+r'$ by $V(ll'a-l'x)+r+V(l)\Square{}{}V(ll'a-l'x')+r'+V(l')$. So, we may assume $l=l'$. If $\gamma=V(la-x)\in VM$, then
$$V(la-x)+r\Square{}{}V(la-x')+r'\mbox{ iff } V(la-x)=\gamma\wedge \gamma+r\Square{}{}V(la-x')+r'.$$
Then by choice of $b$, $V(lb-h(x))=h(\gamma)\wedge h(\gamma)+r\Square{}{}V(lb-h(x'))+r'$. So we are done. The case $V(la-x')\in VM$ is similar.
 If $V(la-x)\notin VM$ and $V(la-x')\notin VM$. Then, $V(la-x')=\min\{V(la-x), V(x-x')\}= V(la-x)$ as $V(x-x')\in VM$. So, 
$$V(la-x)+r\Square{}{}V(l'a-x')+r'\mbox{ iff } r\Square{}{}r' \wedge V(la-x')<V(x-x')\wedge V(la-x)<V(x-x').$$
Then, by choice of $b$, $V(lb-h(x'))<V(h(x)-h(x'))\wedge V(lb-h(x))<V(h(x)-h(x'))$. So, $ V(lb-h(x))+r\Square{}{}V(l'b-h(x'))+r'$ holds as $V(lb-h(x))=V(lb-h(x'))$.

\end{proof}

The second part of Theorem \ref{EQ valued groups} follows from the above proposition:
\begin{proof}
\par Let $M',N'$ be saturated models of $T_{pV,N}$ and $M,N$ be small substructures and $h$ be an isomorphism between $M$ and $N$. By the choice of our language $M$ and $N$ are pure subgroups of $M'$ and $N'$. Let $a\in M'\setminus M$. Let $p$ be its quantifier-free type over $M$. First remark that $p$ does not contain a formula of the form $kX-b=0$ with $k\in \Z$ and $b\in M$. Indeed, if this is not the case then $ka=b$ and as $M$ is a pure subgroup of $M'$ and $M'$ is torsion-free, $a\in M$: contradiction. Next, we observe that any formula in $p$ is equivalent to a formula like in Proposition \ref{extension isom}: indeed, a formula of the form $\neg V(kX-b)\leq \gamma$ is equivalent to $V(kX-b)> \gamma$. Similarly for any formula that involves a negation attached to a  valuation (in)equality. Finally, $\neg kX-b=0$ is equivalent to $V(kX-b)<\infty$.
\par So, by Proposition \ref{extension isom}, there is $b\in N'$ a realisation of the image of $p$ by $h$ and $h$ extends to an isomorphism between $M\langle a\rangle$ and $N\langle b\rangle$. As $M\langle a\rangle$ and  $N\langle b\rangle$ are small pure subgroup of $M$ and $N$, we obtain a back-and-forth system. This completes the proof of quantifier elimination.
\end{proof}

\par Before we end this section, we prove that the theory of $p$-valued groups is NIP whenever all indices $[n]G$ are finite. We refer to \cite{Simon} for definitions and properties of NIP theories.
\begin{theorem}\label{p-valued group are NIP}
 For all $N\in \N^\N$,  $T_{pV, N}$ is NIP. In particular,  $Th((\Z, +,-,0,1, \equiv_n), (\N\cup\{\infty\}, 0,\infty, S), v_p)$ is NIP.
\end{theorem} 
 The case $G=\Z$ has been proved by F. Guignot using an argument of coheir counting in \cite{Guignot-thesis}. We propose here a proof using sequences of indiscernibles, which is essentially the same as the case of valued fields with NIP theory (see \cite{Simon} Appendix A for instance).
\begin{proof}
Let $\mathcal{M}=(M, vM)$ be a model of the theory.
Let $(x_i, i<\omega)$ be a sequence of indiscernibles. We have to prove that for all $\overline{y}\subset \mathcal{M}$ and for all $\Phi$ formula, either $\mathcal{M}$ satisfies $\Phi (x_i,\overline{y})$ for all $i$ large enough or  $\mathcal{M}\vDash \neg\Phi (x_i,\overline{y})$ for all $i$ large enough. By quantifier elimination and properties of NIP theory, we may assume that $\Phi$ is of the type
\begin{enumerate}[(1)]
	\item $\varphi(x_i,\overline{y})$ a $(+,-,0,1, \equiv_n)$-formula and $x_i,\overline{y}\subset M$;
	\item $\varphi(x_i,\overline{y})$ a $(0,\infty, S,<)$-formula and $x_i,\overline{y}\subset VM$;
	\item $\varphi(V(P_1(x_i,\overline{y})), \cdots, V(P_n(x_i,\overline{y})),\overline{z})$ a $(0,\infty, S,<)$-formula and $x_i,\overline{y}\subset M$, $\overline{z}\subset VM$ and $P_i$ is a $\Z$-linear combination.
\end{enumerate}
In case (1) and (2), we are done because the theory of each sort is NIP. We reduce now case (3) to case (2):

\par Let $P(X,\overline{Y})= nX+\sum m_iY_i$ with $n,m_i\in \Z$. We claim that there is $(a_i<\omega)$ sequence of indiscernible in the second sort $VM$ such that for all $i$ large enough, $V(P(x_i,\overline{y}))=S^k(a_i)$ for some $k\in \Z$ (which depends only on $n$) or $V(P(x_i,\overline{y}))$ is eventually constant: Set $B=\sum m_i y_i$. Then $V(P(x_i,\overline{y}))=V(nx_i+B)$.
\begin{itemize}
	\item First, if $v(x_i)$ is not constant, then by indiscernibility, it is either strictly increasing or strictly decreasing. So, for all $i$ large enough, $V(nx_i+B)=V(nx_i)$ or $V(nx_i+B)=V(B)$. Take $a_i=V(x_i)$.
	\item Otherwise, let $t_i=x_i-x_0$. If $V(t_i)$ is not constant, it has to be decreasing. For assume it is strictly increasing. Then, for all $i>1$, $V(x_i-x_1)=V(t_i-t_1)=V(t_1)$. Therefore, the sequence $V(x_i-x_1)$ is constant while the sequence $V(x_i-x_0)$ is increasing: this contradicts the indiscernibility of the sequence $(x_i)$. Now, $V(nx_i+B)=V(nx_i-nx_0 + B-nx_0)$. The latter is equals to $V(B-nx_0)$ or to $V(nx_i-nx_0)=V(nt_i)$ for all $i$ large enough. Take $a_i=V(x_i-x_0)$ in the second case.
	\item Finally, assume that $V(t_i)$ is constant. Then, by properties of $p$-valued $\Z$-groups, $V(at_i-t_1)>V(t_1)$ for some $1\leq a<p$ (independent of $i$ by indiscernability). As $V(at_i-t_1)>V(t_1)$ and $V(at_j-t_1)>V(t_1)$, 
	$$V(t_i-t_j)=V(at_i-at_j)\geq \min \{V(at_i-t_1),V(at_j-t_1)\}> V(t_1)=V(t_i).$$
		Let $x_\omega$ such that $(x_i,i\leq \omega)$ is indiscernible. As $V(x_2-x_1)=V(t_2-t_1)>V(t_2)=V(x_2-x_0)$, the sequence $V(x_\omega-x_i)$ is increasing by indiscernibility. Take $a_i=V(x_\omega-x_i)$ like in the second case.
\end{itemize}
In all cases, either, $V(nx+B)$ is constant or it is equal to $a_i+v_p(n)=S^{v_p(n)}(a_i)$.

\end{proof}

\section{Expansion of $\Qp$ by a multiplicative subgroup}\label{Expansion of Qp by a multiplicative subgroup}
In this section, we assume $p\not=2$ (the case $p=2$ is similar but require technical changes as in the remark below).
\begin{lemma}\label{subgroup Qp} Let $G<\Qp^*$. Then $G= T^\Z\times D$, where $T\in G$, $D<\Zp^\times$ and $D^s$ is either trivial or dense in $1+p^n\Zp$ for some $n\in\N$ and $s$ that divides  $p-1$.
\end{lemma}
\begin{proof}
\par First, let $T\in G$ with minimal positive valuation among the valuations of the elements of $G$. If such an element does not exist, then $G<\Zp^\times$ and we take $T=1$. If $T\not=1$, for all $t\in G$, there is $n\in \Z$ such that $v(t)=nv(T)$. Otherwise, if $t$ has positive valuation, there is $n$ such that $nv(T)<v(t)<(n+1)v(T)$ i.e. $0<v(t/T^n)<v(T)$: this contradicts the minimality of $v(T)$. If $t$ has negative valuation, we replace $t$ by $t^{-1}$. Then, $t=T^nu$ with $u\in G\cap \Zp^\times$, $n\in \Z$ (take $u=t/T^n$). 
\par Now, we may assume that $G<\Zp^\times$. Let us recall that  $\Zp^{\times}\cong \mu_{p-1}\times (1+p\Zp)$ where $\mu_{p-1}$ is the set of $(p-1)$th roots of unity (that is isomorphic to $\Fp^*$ via the residue map $res$). So, $res(G)$ is a subgroup of $\Fp^*$. Let $s$ be its order. Then, for all $g\in G$, $res(g^s)=1$ i.e. $G^s<1+p\Zp$.
\par Finally, let us remind that $1+p^n\Zp$ is isomorphic to $p^n\Zp$ (as groups): 
 $exp_p:p^n\Zp\rightarrow 1+p^{n}\Zp:x\longmapsto \sum_k x^k/k!$ is an isomorphism for all $n>0$ whose inverse is determined by the $p$-adic logarithm map $\log_p$ (given by the usual power series $\sum_{k} (-1)^{k+1}x^k/k$).
Furthermore, $\log_p:(1+ p^n\Zp,\cdot)\rightarrow (p^n\Zp,+)$ is a bicontinuous isomorphism of groups. So, $G':=\log_p G$ is an additive subgroup of $\Zp$. If $G'=\{0\}$ then $G^s=\{1\}$ and $G=\mu_s$. Otherwise, let $n$ be the minimal valuation of the elements of $G'$. Then, let $g\in G'$ with $v(g)=n$. As $g\Z$ is dense in $p^n\Zp$, $G'$ is a dense subset $p^n\Zp$ (by minimality of $n$). So, $G=exp_p(G')$ is a dense subset of $1+p^n\Zp$.
\end{proof}
\begin{Remark} If $p=2$, $\Z_2^\times= 1+2\Z_2$. The exponential is well-defined on $1+4\Z_2$. So, we obtain that $D^2$ either trivial or dense in $1+4\Z_2$.
\end{Remark}
\par First, we deal with the discrete case i.e. when $D$ is a finite group i.e. $D=\mu_s<\Fp^*$. This case is similar to $(\R, 2^{\Z})$ in \cite{vdDries5}.
Let $G=T^\Z\times D<\Qp^*$ with $v_p(T)>0$. Then, the valuation induces an isomorphism between $T^\Z$ and $v((\Qp^k)^*)$ where $k=v_p(T)$. We denote by $\lambda$ the map $\Qp[k]\rightarrow T^\Z:x\longmapsto T^n$ where $n=v_p(x)/k$ and $\Qp[k]$ is the set of elements in $\Qp^*$ with valuation divisible by $k$. 
We work in the language $\mathcal{L}_G=\mathcal{L}_{Mac}\cup \{A,\gamma_T,r,\lambda \}$ where $A$ will be interpreted in $\Qp$ by the subgroup $G$, $\gamma_T$ by $T$, $r$ by a (fixed) $s$th root of unity $\xi$ and $\lambda$ by the above map. Let $K$ be a $p$-adically closed field. Then, let $K[k]$ denote the subgroup of $K^*$ of elements of valuation divisible by $k$. Let us remark that this group is $\mathcal{L}_{Mac}$-definable.
\begin{theorem}\label{discrete case} $T_d(G)=Th(\Qp, +,-,\cdot, 0,1, P_n (n\in \N), G, T, \xi, \lambda)$ admits the elimination of quantifiers. Furthermore, $T_d(G)$ is axiomatised by:  if $(K,+,-,\cdot, 0,1, P_n (n\in \N), A, \gamma_T, d, \lambda)$ is a model of $T_d(G)$ then
\begin{itemize}
	\item $(K,+,-,\cdot, 0,1, P_n (n\in \N))$ is a $p$-adically closed field;
	\item $(A,\cdot, 1)$ is a subgroup of $K[k]$, $\gamma_T\in A$, $A_{tor}=\mu_s=\{1,r,\cdots, r^{s-1}\}$;
	\item $\forall x\in K[k] \left(x\not=0\rightarrow \big[\exists y\in A v_K(x)=v_K(y)\right.$
	$\left.\wedge (\forall z\in A v_K(z)=v_K(x)\rightarrow \vee_{0\leq i<s} z=r^iy)\big]\right)$ where $k=v_p(T)$;
	\item $\lambda: K[k]\rightarrow A$ is a morphism of groups, $\lambda(p^k)=\gamma_T$, for all $z\in K[k] v(\lambda(z))=v(z)$ and $A=\lambda(A)\times A_{tor}$, $A_{tor}\cap \lambda(A)=\{1\}$;
	\item $tp(\gamma_T/\Q)=tp(T/\Q)$, $r^s=1$, $res\ r=res\ \xi$.
\end{itemize}
\end{theorem}
\begin{Remark}  In the above theorem, the valuation $v_K$ is not part of the language but the above axioms are first-order as for $p$-adically closed fields: $v_K(x)\geq v_K(y)$ iff $y^2+px^2$ is a square.
\end{Remark}
\par This theorem is a corollary of Theorem 1.1 and Theorem 1.2 in \cite{Mariaule}. These theorems give an axiomatisation and quantifier elimination in the special case $T_d(T^{\Z})$. As $G=\mu_sT^\Z$ is definable in $(\Qp, +, \cdot, 0, 1, T^{\Z}, (P_n)_{n\in \N})$, we obtain the above result. Note that the proof of the theorem of gives the following if $K, L\vDash T_d(G)$ have isomorphic value groups, then there is an isomorphism of $\mathcal{L}_G$-structures between $\Q(A(K))^h$ and $\Q(A(L))^h$ (note that both are models of $T_d(G)$). See also Theorem \ref{EQ Log} in Section \ref{Expansion for subanalytic structure}.
\par The above theorem treats the case where $D$ is finite. Let us now deal with the case $s=1$ i.e. $G=T^\Z\times D$ for some $D$ dense subgroup of $1+p^n\Zp$.
 So, $\log_p(D)$ is an additive subgroup of $\Zp$ and therefore $(\log_p(D),v_p)$ is a abelian $p$-valued group. On the other hand, $v_p(exp_p(x)-1)=v_p(x)$ and $v_p(\log_p(1+x))=v_p(x)$. We define $V:D\rightarrow \N\cup\{\infty\}:x\longmapsto v_p(x-1)-n$. Then, $(D,V)$ is an abelian $p$-valued group isomorphic (as valued group) to $(\log_p(D),v_p)$. We will give now an axiomatisation of $Th(\Qp, G)$. First, let us introduce some notions required for our axiomatisation.

\par Let $K$ be a field of characteristic zero and $G$ be a subgroup of $K^*$. Let $a_1,\cdots, a_n\in \Q$ nonzero. We consider the equation
$$a_1x_1+\cdots+a_nx_n=1.$$
A solution $(g_1,\cdots, g_n)$ of this equation in $G$ is called nondegenerate if $\sum_{i\in I}a_ig^i\not=0$ for all $I\subset\{1,\cdots, n\}$ nonempty. We say that $G$ has the \emph{Mann property} if for any equation like above there is finitely many nondegenerate solutions in $G$. Examples of groups with Mann property are the roots of unity in $\C$ \cite{Mann} or any group of finite rank in a field of characteristic zero \cite{Evertse}\cite{Laurent}\cite{vanderPoorten}. In particular, any subgroup of $\Qp^*$ of finite rank has the Mann property.
\par Let $G<K^*$ be a group with the Mann property. Then the Mann axioms are axioms in the language of rings expanded by constant symbols $\gamma_g$ for the elements of $G$ and a unary predicate $A$ for $G$. Let $a_1,\cdots , a_n\in\Q^*$. As $G$ has the Mann property, there is a collection of $n$-uples $\overline{g}_i=(g_{1i},\cdots, g_{ni})$  ($1\leq i\leq l$) in $G^n$ so that these $n$-uples are the nondegenerate solutions of the equation $a_1x_1+\cdots+a_nx^n=1$. The corresponding Mann axiom express that there are no extra nondegenerate solutions in $A$ i.e.
$$\forall\overline{y}\left[\left(\bigwedge_i A(y_i)\wedge\sum_{i=1}^n a_iy_i=1\wedge\bigwedge_{I\subset \{1,\cdots ,n\}}\sum_{i\in I} a_iy_i\not=0\right)\rightarrow \bigvee_{k=0}^l \overline{y}=\overline{\gamma}_{g_k}\right]. $$
The main consequence of Mann axioms that we will use is the following:
\begin{lemma}[Lemmas 5.12 and 5.13 in \cite{Gunaydin-vdD}]\label{G-vdD Lemma 5.12} Let $K$ be a field of characteristic zero, let $G$ be a subgroup of $K^*$ and let $\Gamma$ be a subgroup of $G$ such that for all $a_1,\cdots a_n\in \Q^*$ the equation $a_1x_1+\cdots +a_nx_n$ has the same nondegenerate solutions in $\Gamma$ as in $G$. Then, for all $g, g_1,\cdots, g_n\in G$
\begin{itemize}
	\item if $g$ is algebraic over $\Q(\Gamma)$ of degree $d$ then $g^d\in \Gamma$;
	\item if $g_1,\cdots, g_n$ are algebraically independent over $\Q(\Gamma)$ then they are multiplicatively independent over $\Gamma$.
\end{itemize} 
In particular, if $\Gamma$ is  a pure subgroup of $G$, then the extension $\Q(G)$ over $\Q(\Gamma)$ is purely transcendental.
\end{lemma}
 
\par Let $\mathcal{M}=(M,\cdots)$ be a $\mathcal{L}$-structure. Let $A\subset M$ and $\mathcal{L}_A$ be the expansion of $\mathcal{L}$ by a unary predicate that will be interpreted by $A$ in $M$. We denote by $f: X\stackrel{n}{\rightarrow} Y$ a map from $X$ to the subsets of $Y$ of size at most $n$. We say that $A$ is \emph{large} in $M$ if there is a $\mathcal{L}_A$-definable map $f:M^m\stackrel{n}{\rightarrow} M$ such that $f(A)=\bigcup_{x\in A^m} f(x)=M$. We say that $G$ is \emph{small} if it is not large. If $G$ is a subgroup of $\Qp^*$ with finite rank then it is small (because of the respective cardinalities of $\Qp$ and $G$). Let us remark that smallness can be written as a scheme of first-order sentences in the language $\mathcal{L}_A$.

\par Let $G<\Qp^*$ of the type $T^\Z\times D$ where $T\in G$ and $D$ is dense in $1+p^n\Zp$ as in Lemma \ref{subgroup Qp} (case $s=1$). We assume that $G$ has the Mann property and is small (for instance $G$ is finitely generated). We will give an axiomatisation of $Th(\Qp, G)$. We assume that $T\not=1$. The changes that have to be made in the case $T=1$ are obvious. Let $\mathcal{L}_G=\mathcal{L}_{Mac}\cup\{ A, \lambda, \equiv_n (n\in\N), \gamma_g (g\in G)\}$. We define the theory $T_G$ as follow: Let $(K, +,-,\cdot, 0,1, P_n (n\in\N), A, \lambda, \equiv_n (n\in\N), \gamma_g (g\in G))$ be a model of $T_G$, then
\begin{itemize}
	\item $(K, +,-,\cdot, 0,1, P_n (n\in\N))$ is a $p$-adically closed field;
	\item $A$ is a multiplicative subgroup of $K^*$;
	\item Let $k=v_p(T)$. Then, for all $x\in K[k]$, there is $a\in A$ such that $v_K(x)=v_K(a)$. $\lambda:K[k]\rightarrow A:x\longmapsto a$ with $v_K(x)=v_K(a)$ is a morphism groups, $\lambda(x)=1$ for all $x$ with $v(x)=0$ and $\lambda(p^k)=\gamma_T$;
	\item For all $x\in A$, there is a unique $x'\in (A\cap \mathcal{O}_K^\times)=:A_V$ such that $x=\lambda(x)x'$;
	\item $(K, +,-,\cdot, 0,1, P_n (n\in\N), \lambda(A), \gamma_T,1,\lambda)$ is a model of $T_d(T^{\Z})$;
	\item $((A_V, \cdot, ^{-1},1,\gamma_e, \equiv_n (n\in\N)), (VK_{\geq 0}, < ,S,0,\infty),\gamma_g (g\in G), V)$ is elementary equivalent to $(D, (\gamma_g)_{g\in G} V)$ as abelian $p$-valued groups where $V:A_V\rightarrow vK^*_{\geq 0}\cup\{\infty\}:a\longmapsto v_K(a-1)-n$ (surjective map), $\gamma_e$ is a fixed element of $G$ with minimal $V$-valuation and $a\equiv_n b$ iff there is $z\in A$ such that $a=bz^n$;
	\item $A_V$ is a dense subset of $1+p^n\mathcal{O}_K$; 
	\item $A$ satisfies the same Mann axioms as $G$;
	\item $A$ is a small subset of $K$;
	\item $A_V\vDash Diag(G/\Q)$.
\end{itemize} 
\begin{Remark} As the valuation is interpretable in our language, we can write the condition that $A_V\equiv D$ (as valued group). For note that the theory $T_{pV,N}$ as defined in section \ref{valued groups} is interpretable and its expansion by the diagram of $D$ is complete by Theorem \ref{EQ valued groups}.
\end{Remark}

\begin{theorem}\label{completness} $T_G$ is complete.
\end{theorem}
The proof is similar to Theorem 7.1 in \cite{Gunaydin-vdD} or Theorem 3.3 in \cite{Mariaule}. As in these proofs, we use the notions of free, linearly disjoint and regular extension of fields. We refer to \cite{Lang} for their definitions and elementary properties.

\begin{proof}
Let $(K,A)$ and $(L,B)$ be two saturated models of $T_G$ of same cardinality ( the saturation is at least $|G|$). Let $Sub(K,A)$ be the collection of $\mathcal{L}_{Mac}\cup \{A\}$-substructures $(K',A')$ where
\begin{itemize}
	\item $K'$ is a $p$-adically closed field, $|K'|<|K|$;
	\item $A'$ is a pure subgroup of $A$;
	\item $K'$ and $\Q (A)$ are free over $\Q (A')$.
\end{itemize} 
We define $Sub(L,B)$ similarly. Let $\Gamma$ be the set of  $\mathcal{L}_{Mac}\cup \{A\}$-isomorphisms of $K'\in Sub(K,L)$ and $L'\in Sub(L,B)$ that extends to an isomorphism of valued fields between $K'(\lambda(K))^h$ and $L'(\lambda(L))^h$. Note that this latter isomorphism is also an isomorphism of $\mathcal{L}_G$-structures.
\begin{claim} $\Gamma$ is nonempty
\end{claim} 
For let us remark that $(\Q(G)^h, G)$ belongs to the sets $Sub(K,A)$ and $Sub(L,B)$ (this follows from the axioms and Lemma \ref{G-vdD Lemma 5.12}). Now let us remark that $(\Q(G)^h, \lambda(G))$, $(K,\lambda(K))$ and $(L,\lambda(L))$ are all models of  $T_d(\lambda(G))$. As  $(\Q(G)^h, \lambda(G))$ is a substructure of $(K,\lambda(K))$ and $(L,\lambda(L))$ resp., by the proof of Theorem \ref{discrete case}, there is an isomorphism between $\Q(G)(\lambda(K))^h$ and $\Q(G)(\lambda(L))^h$ (extending the indentity).

\begin{claim}\label{sub-pseudo-Cauchy} Let $K'$ $p$-adically closed fields with $|K'|<|K|$. Let $x\in K\setminus K'(\lambda(K))^h$. Then the type $tp(x/K'(\lambda(K))^h)$ is realised in any $|K'|$-saturated expansion of $K'$.
\end{claim}
\par For as $K$ is an immediate expansion of $K'(\lambda(K))^h$, the type of $x$ over $K'(\lambda(K))^h$ is determined by a pseudo-Cauchy sequences over $K'(\lambda(K))$. So, it is sufficient to find a sequence with same pseudo-limit and indexed over a set of size at most $|K'|$.
\par Let ${K'}_{\lambda}^{(1)}:=K'(\lambda(K'(x)))$ and by induction ${K'}_{\lambda}^{(n+1)}:={K'}_{\lambda}^{(n)}(\lambda({K'}_{\lambda}^{(n)}))$. Set ${K'}_{\lambda}^{\omega}:=\bigcup_n {K'}_{\lambda}^{(n)}$. Let us remark that $|{K'}_{\lambda}^{\omega}|=|K|$.
 \par Now it is sufficient to prove that $\{v(x-b):b\in {K'}_{\lambda}^{\omega}\}$ is cofinal in $\{v(x-f(\overline{a}, \overline{g})): \overline{a}\subset K', \overline{g}\subset \lambda(K)\}$ for all $f\in \Z[\overline{X},\overline{y}]$. Let $f(\overline{g}):=f_0+\sum f_I\overline{g}^I$ with $f_I\in {K'}_{\lambda}^{(n)}$ (in particular, any polynomial $f(\overline{a},\overline{g})$ in the last set has this form with $n=1$). We prove by induction on the number $d$ of nonzero $f_I$ that either $v(x-f_0-\sum f_I\overline{g}^I)\leq v(x-b)$ for some $b\in  {K'}_{\lambda}^{(n+1)}$ or there are $d-1$ non zero $f'_I\in {K'}_{\lambda}^{(n+1)}$ such that $f_0+\sum f_I\overline{g}^I=f'_0+\sum f'_I\overline{g}^I$. This proves our claim by induction on $d$ the degree of $f$.
\par  If $f(\overline{g})$ has only one monomial $f_I\overline{g}^I$ then $x-f(\overline{g})=x-f_0-f_I\overline{g}^I$. Then either $v(x-f_0)\not=v(b_I\overline{g}^I)$ so $v(x-f_0-f_I\overline{g}^I)=\min\{v(x-f_0), v(f_I\overline{g}^I)\}\leq v(x-f_0)$. In that case, take $b=f_0$. Or $v(x-f_0)=v(f_I\overline{g}^I)$ i.e. $\lambda(x-f_0)=\lambda(f_I)\overline{g}^I$. Therefore, $x-f_0-f_I\overline{g}^I= x-f_0-f_I \lambda(f_I)^{-1}\lambda(x-f_0)$. As $f_0+f_I \lambda(f_I)^{-1}\lambda(x-f_0)\in {K'}_{\lambda}^{(n+1)}$ we are done: take $b=f_0+f_I \lambda(f_I)^{-1}\lambda(x-f_0)$.
\par For the inductive step, let $f(\overline{g})=f_0+\sum_{I\not=0} f_I\overline{g}^I$ where $f_I\in {K'}_{\lambda}^{(n)}$ with $d$ nonzero factors $f_I\overline{g}^I$. First let us remark that if $v(f_I\overline{g}^I)\not=v(f_J\overline{g}^J)$ and $v(x-f_0)\not=v(f_I\overline{g}^I)$ for all $I\not=J$ with $f_I,f_J$ nonzero then $v(x-f_0+\sum_{I\not=0} f_I\overline{g}^I)=\min\{v(x-f_0),v(f_I\overline{g}^I)\}\leq v(x-f_0)$. In that case, we take $b=f_0$. Otherwise, if there is $I\not= J$ such that $f_I,f_J\not=0$ and $v(f_I\overline{g}^I)=v(f_J\overline{g}^J)$ i.e. $\lambda(f_I)\overline{g}^I=\lambda(f_J)\overline{g}^J$. Set $f'_I=f_I+f_J\lambda(f_I)\lambda(f_J^{-1})$, $f'_J=0$ and $f'_K=f_K$ for $K\not=I,J$. Then, $f_0+\sum f_I\overline{g}^I=f'_0+\sum f'_J\overline{g}^J$, $f'_I\in {K'}_{\lambda}^{(n+1)}$ for all $I$ and there are $d-1$ nonzero $f'_I$. Finally, if $v(x-f_0)=v(f_I\overline{g}^I)$, we proceed similarly substituting $f_0$ by $f_0+f_I\lambda(x-f_0)\lambda(f_I^{-1})$. This completes the proof of Claim \ref{sub-pseudo-Cauchy}.

\par We will prove that $\Gamma$ forms a back-and-forth system. So, let $\iota: (K',A')\rightarrow (L',B')$ in $\Gamma$ and $\widetilde{\iota}$ its extension to $K'(\lambda(A))$. Let $\alpha\in K\setminus K'$. We construct an extension of $\iota$ that contains $\alpha$ in its domain and compatible with $\widetilde{\iota}$. There are different cases according to the choice of $\alpha$.
	\par 1. If $\alpha\in A$ then $\alpha=\lambda(\alpha)\alpha'$ for some unique $\alpha'\in A\cap \mathcal{O}_K$. So, we may assume that
		\par (a) $\alpha=\lambda(\alpha)$. As $\alpha\notin K'$, $v_K(\alpha)\notin v(K')$. Let $l\in v_K(L)$ such that it realises the same cut as $v_K(\alpha)$ over $v_K(K')$ and the same congruence classes in $v_K(L)$ ($l$ exists by saturation and quantifier elimination for $\Z$-groups). Then, as $v_K(\alpha)$ is $k$-divisible, $l$ is $k$-divisible. Let $x\in L$ such that $v_L(x)=l$. Take $\beta=\lambda(x)$. Let $K''=K'(\alpha)^h$, $A''=A\cap K''$, $L''=L'(\beta)^h$, $B''=B\cap L''$. $\iota$ extends to an isomorphism between $(K'',A'')$ and $(L'',B'')$ with $\alpha\longmapsto \beta$ (that is the reduct of $\widetilde{\iota}$). Also, by Lemma \ref{G-vdD Lemma 5.12} and properties of $p$-adically closed fields, $(K'',A'')\in Sub(K,A)$ and $(L'',B'')\in Sub(L,B)$.  This follows from the following fact:
		\begin{Fact}\label{alg indep} For all $t_1,\cdots , t_n\in K'$ algebraically independent over $A$, $$acl(t_1,\cdots, t_n,\alpha, A')\cap A=A'\langle\alpha\rangle$$ (here $acl$ is taken in the language of rings, $A'\langle\alpha\rangle$ is the pure closure of $A'(\alpha)$ in $A$ as defined in Definition \ref{pure subgroup}).
		\end{Fact}
		This is a consequence of Mann property (in particular, Lemma \ref{G-vdD Lemma 5.12}) and exchange property for $acl$, see \cite{Belegradek-Zilber} proof of Lemma 4.2 for instance.
		In particular, for all $t\in K''$, $t\in A$ iff $t\in A'\langle\alpha\rangle$ (so, $A''=A'\langle\alpha\rangle$)

		\par (b) $\alpha\in (A\cap \mathcal{O}_K)=:A_V$.  We want to find $\beta$ in $B_V$ that realises the same type as $\alpha$ over $A'_V$ (type in the theory of $p$-valued group) and the same type as $\alpha$ over $K'(\lambda(K))^h$ (by Claim \ref{sub-pseudo-Cauchy}, this latter is determined by the set of formulas $v_K(x-a)\Square{}{} v_K(t)$ with $a,t$ in a subfield of $K'(\lambda(K))$ of size $|K'|$). In that case, $A'_V\langle\alpha\rangle\cong B'_V\langle \beta\rangle$ as valued field and $K'(\lambda(A))(\alpha)^h\cong L'(\lambda(B))(\beta)^h$ as valued fields. Let us prove that the union of these type is finitely consistant i.e. that given finitely many conditions $V(x-g)\Square{}{}\gamma$, $x\equiv_n g$ with $g\in A'$ and $\gamma\in VA'$ and a ball $B$ in $K'(\lambda(K'(\alpha)))$ the image under $\widetilde{\iota}$ of these formulas is realised. Without loss of generality, we may assume that $B$ is a subset of the set of $x\in K'$ such that $V(x-g)\Square{}{}\gamma$. By Proposition \ref{extension isom} and its proof, there is $\beta\in B_V$ that realises the formulas $V(x-\iota(g))\Square{}{}\widetilde{\iota}\gamma$ and $x\equiv_n \iota(g)$. By density of $B_V$ in $1+p^n\mathcal{O}_K$, we may even assume that such realisation is in $\widetilde{\iota}(B)$. So the image by $\widetilde{\iota}$ of the type is consistant. Let $\beta$ be one of its realisation in $L$. We conclude like in case 1.(a) setting $K''=K'(\alpha)^h$, $A''=A\cap K''$, $L''=L'(\beta)^h$, $B''=B\cap L''$.
	\par 2. If $\alpha\in K'(\alpha_1,\cdots, \alpha_n)^h$ where $\alpha_i\in A$, then apply case 1. $n$ times.
	\par 3. If $\alpha\notin K'(A)^h$.
	Consider the cut of $\alpha$ over $K'(\lambda(K))$. By saturation, Claim \ref{sub-pseudo-Cauchy} and smallness of $A$,  we can find $\beta\in L\setminus L'(B)^h$ which realises the corresponding cut (under $\widetilde{\iota}$). Then, $K'(\lambda(K))(\alpha)^h\cong L'(\lambda(L))(\beta)^h$. Take $K''=K'(\alpha)^h$ and $L''=L'(\beta)^h$. Note that as $\Q(A)$ is a regular extension of $\Q(A')$, $K'$ and $\Q(A)$ are linearly disjoint over $\Q(A')$. By linear disjointness, $\iota$ extends to an isomorphism between $(K'',A')$ and $(L'',B')$ that sends $\alpha$ on $\beta$. The freeness of $K''$ and $\Q(A)$ over $\Q(A')$ follows from the assumption that $\alpha\notin K'(A)^h$.
\end{proof}

Finally, let us return to the general case: let $G$ be a subgroup of $\Qp$. In Lemma \ref{subgroup Qp}, we have seen that $G=T^{\Z}\times D$ where $D^s$ is dense in $1+p^n\Zp$. Let us remark that any $t\in D$ is uniquely determined by the pair $(t^s, res\ t)$. Indeed, for all $h\in D^s$ and $\xi\in \Fp^*$, there is by Hensel's Lemma at most one $t\in D$ such that $t^s=h$ and $res\ t=\xi$.
Furthermore, if $t'\in D$ is such that $t'^s=h$ then $t'=t\omega$ for some $\omega\in G_{tor}$.
\par Let us return to the back-and-forth system of Theorem \ref{completness}. The general case does not change neither case 1. (a), case 2. nor case 3. It remains to deal with case 1.(b). For let us remark that given $\alpha\in A_V$, we can find $\beta^s\in B_V^s$ such that ${A'}_V^s\langle\alpha^s\rangle\cong {B'}_V^s\langle\beta^s\rangle$ (where the pure closure is taken in $A^s$, $B^s$ resp.) and $K'^s(\alpha^s)\cong L'^s(\beta^s)$ as valued fields (this latter isomorphism is compatible with $\widetilde{\iota}$). This follow from the proof of case 1.(b). It remains to see if this isomorphism is an isomorphism of $\mathcal{L}_G$-structures. For it is sufficient to prove that for all $x\in K$, $x\in {A'}_V\langle\alpha\rangle$ iff $\iota(x)\in {B'}_V\langle\beta\rangle$. For by the above remark, $x\in {A'}_V\langle\alpha\rangle$ is determined by the properties $x^s\in {A'}_V^s\langle\alpha^s\rangle$ and $res\ x=\xi$. If $x\in {A'}_V^s\langle\alpha^s\rangle$ say $x^n=a\alpha^m$ for some $a\in {A'}_V^s$ and $n,m\in \N$. Then $\alpha$ satisfies $\alpha^m\in a^{-1}.D_n(\xi)$ where $D_n(\xi):=\{t\in D^s:\ t\equiv_n0,\ \exists y\in D y^s=t\mbox{ and } res\ y=\xi\}$ (where the congruence is in the group $D^s$). If $n=1$, we set $D_n(\xi)=:D(\xi)$. Therefore, for $\iota$ to be an isomorphism of $\mathcal{L}_G$-structures, it is necessary that $\beta^m\in \iota(a^{-1})D_n(\xi)$ (where $D_n$ is now interpreted in $L$). This will be guaranteed by a scheme of axioms expressing the density of $D_n(\xi)$ and adding this property in the type of $\beta^s$.                       
\par We will see in the next lemma that such set $a^{-1}D_n(\xi)$ is dense in a finite union of ball in $1+p^n\Zp$ (possibly empty). Futhermore, we can add the information $\alpha^s\in a^{-1}D_n(\xi)$ to the type defined in case 1.(b) and still guarantee that this type is finitely satisfied in $L$. In that case we are done.

\begin{lemma}\label{density of roots} If $D(\xi)$ is nonempty, then it is dense in $1+p^n\Zp$. Futhermore, for all $x_1,\cdots, x_m\in 1+p^n\Zp$, for all $k_1,\cdots, k_m,n_1,\cdots, n_m\in \N$ if there is $y\in D$ such that $y^{k_i}\in x_i^{-1}D_{n_i}(\xi_i)$ for all $i$, then the set of $z\in D$ such that $z^{k_i}\in x_i^{-1}D_{n_i}(\xi_i)$ for all $i$ is dense in a finite union of (multiplicative) cosets of $1+p^r\Zp$ in $1+p^n\Zp$ for some $r\geq n$ where $r$ is independent of $x_i$. If $x_i\in D^s$ the cosets depend only on the congruence classes of $x_i$ modulo $s+\sum n_i$ (in $D^s$).
\end{lemma} 
\begin{proof}
Let us proof the first part of the lemma. Let $y\in D(\xi)$ i.e. $y=x^s$ for some $x\in D$ with $res\ x=\xi$. First let us remark that $yD^{s^2}\subset D(\xi)$. For if $z\in D^{s^2}$ then $z=t^s$ for some $t\in D^s$. So $yz=(xt)^s$ and  $res\ xt=res\ x\cdot res\ t=res\ \xi$ (as $res\ t=1$ since $t\in D^s\subset 1+p^n\Zp$) i.e. $yz\in D(\xi)$. Now, it remains to see that $yD^{s^2}$ is a dense subset of $1+p^n\Zp$. That is $\log_p(yD^{s^2})=\log_p(y)+s\log D^s$ is a dense subset of $p^n\Zp$. This is the case as $v_p(s)=0$, $v_p(\log_p(y))\geq 0$ and $\log_p D^s$ is dense in $p^n\Zp$.

\par The general case is similar. Let $X:= \{y\in D^s:\ \wedge_i y^{k_i}\in x_i^{-1}D_{n_i}(\xi_i)\}$. We assume that $X$ is nonempty. Let $y_0\in X$. For all $i$, $y_0^{k_i}=x_{i}^{-1}(t_{0,i}^s)^{n_i}$ with $res\ t_{0,i}=\xi_i$. Let $X_0:=\{z\in D^s:\ z^{k_i}\in D_{n_i}(1)\}$. Then $X=y_0X_0$. For let $y\in X$ i.e. $y^{k_i}=x_{i}^{-1}(t_{i}^s)^{n_i}$ with $res\ t_{i}=\xi_i$. Then $(y/y_0)^{k_i}=((t_i/t_{0,i})^s))^{n_i}$ with $res\ (t_i/t_{0,i})=1$ i.e. $y/y_0\in X_0$. Conversely, let $z\in X_0$ i.e. $z^{k_i}=(t_i^s)^{n_i}$ with $res\ t_i=1$. Then $(y_0z)^{k_i}=x_i^{-1}((t_{0,i}t_i)^s)^{n_i}$ with $res\ (t_{0,i}t_i)=\xi_i$ i.e. $y_0z\in X$.
\par Let us prove that $X_0$ is a dense subset of a finite union of (multiplicative) cosets of $1+p^r\Zp$ in $1+p^n\Zp$ for some $r\geq n$. Let $z\in X_0$ i.e. for all $i$, $z^{k_i}=(t_i^s)^{n_i}$ where $res\ t_i=1$. Then $zD^{s^2\prod n_i}\subset X_0$. For if $y=t^{s^2\prod n_i}$ then $(zy)^{k_i}= ((t_it^{sk_i\prod_{j\not=i}n_j})^s)^{n_i}$ and $res\ (t_it^{sk_i\prod_{j\not=i}n_j})= res\ t_i\cdot res\ (t^{k_i\prod_{j\not=i}n_j})^s=1$ as $(t^{k_i\prod_{j\not=i}n_j})^s\in D^s\subset 1+p^n\Zp$. Let us remind that $D^{s^2\prod n_i}$ is a dense subset of $1+p^r\Zp$ where $r=n+v_p(s^2\prod n_i)$. Then, $X_0$ is dense in the coset $z(1+p^r\Zp)$. If there is no $y\in X_0$ such that $y\notin z(1+p^r\Zp)$, then $X_0$ is a dense subset of $z(1+p^r\Zp)$. Otherwise, let $y\in X_0$, $y\notin z(1+p^r\Zp)$. Then by the same argument as before $X_0$ is dense in the coset $y(1+p^r\Zp)$. As $(1+p^r\Zp)$ has finitely many coset in $1+p^n\Zp$ and $X_0\subset (1+p^n\Zp)$, we are done.
\par Now, we have that $X=y_0X_0$ is a dense subset of finitely many cosets $y_0c(1+p^r\Zp)$ where $c, r$ does not depends on $x_i$. It remains to prove that these cosets only depends on the on the congruence classes of $x_i$ modulo $s+\sum n_j$. For it is sufficient to prove that for all $z_i\in D^s$ such that $z_i\equiv x_i\mod (s+\sum n_j)D^s$ (i.e. $z_i=x_ib_i^{s\prod n_j}$ (for some $b_i\in D^s)$), if $X(\overline{z})=\{y\in D^s:\ \wedge_i y^{k_i}\in z_i^{-1}D_{n_i}(\xi_i)\}$ then $X=X(\overline{z})$. Let $w\in X(\overline{z})$. Then $w^{k_i}=z_i^{-1}(v_i^s)^{n_i}$ with $res\ v_i=\xi_i$. So $w^{k_i}=(x_ib_i^{s\prod n_j})^{-1}(v_i^s)^{n_i}=(x_i)^{-1}((v_ib_i^{-\prod_{j\not=i} n_j})^s)^{n_i}$ where $res\ v_ib_i^{-\prod_{j\not=i}n_j}=res\ v_i=\xi_i$ as $b_i^{-\prod_{j\not=i}n_j}\in D^s\subset 1+p^n\Zp$. This proves that $w\in X$. By symmetry, $X=X(\overline{z})$.
\end{proof}

\par Let $(K,A)$ be any model of our theory. Let $A_V^s:=(A\cap \mathcal{O}_K^\times)^s$. Let $x_i\in A_V^s$ and $X:= \{y\in  A_V^s:\ \wedge_i y^{k_i}\in x_i^{-1}D_{n_i}(\xi_i)\}$. Then by the above lemma, $X$ is either empty or dense in a finite union of $c(1+p^r\mathcal{O}_K)$ where $r$ is independent of $x_i$ and $c$ only depends on the congruence classes of $x_i$ modulo $s+\sum n_j$.
If all indices $[n]D^s$ are finite, then the congruences classes of $x_i$ modulo $s+\sum n_j$ have representatives in $G$ (or in any pure subgroup of $A$). In that case, the \emph{dense axioms of $D_n(\xi)$} is an scheme of axioms that expresses in our language when $X$ is empty or in which coset of $(1+p^r\mathcal{O}_K)$ it is dense.

\par Let $G<\Qp^*$ be a small group with Mann property and finite index $[n]G$ (for instance a group of finite rank). We fix $T\in G$ such that $T^\Z\times D$, where $D^s$ is dense in $1+p^n\Zp$ as in Lemma \ref{subgroup Qp}. We will give an axiomatisation of $Th(\Qp, G)$. Again, we assume that $T\not=1$. Let $\mathcal{L}_G=\mathcal{L}_{Mac}\cup\{A, \lambda, \equiv_n (n\in\N), \gamma_g (g\in G)\}$. We define the theory $T_G$ as follow: let $(K, +,-,\cdot, 0,1, P_n (n\in\N), A, \lambda, \equiv_n (n\in\N), \gamma_g (g\in G))$ be a model of $T_G$, then
\begin{itemize}
	\item $(K, +,-,\cdot, 0,1, P_n (n\in\N))$ is a $p$-adically closed field;
	\item $A$ is a multiplicative subgroup of $K^*$, $A_{Tor}=G_{Tor}$;
	\item Let $k=v_p(T)$. Then, for all $x\in K[k]$, there is $a\in A$ such that $v_K(x)=v_K(a)$. $\lambda:K[k]\rightarrow A:x\longmapsto a$ with $v_K(x)=v_K(a)$ is a morphism groups, $\lambda(x)=1$ for all $x$ with $v(x)=0$ and $\lambda(p^k)=\gamma_T$;
	\item For all $a\in A$ there is a unique $a'\in \mathcal{O}_K^{\times}$ such that $a=a'\lambda(a)$;
		\item $((A_V^s, \cdot, ^{-1},1,\gamma_e, \equiv_n (n\in\N)), (VK_{\geq 0}, < ,S,0,\infty),\gamma_g (g\in G), V)$ is elementary equivalent to $(D^s, (\gamma_g)_{g\in G} V)$ as abelian $p$-valued groups where $V:A_V^s\rightarrow vK^*_{\geq 0}\cup\{\infty\}:a\longmapsto v_K(a-1)-n$ (surjective map), $\gamma_e$ is a fixed element of $G$ with minimal $V$-valuation and $a\equiv_n b$ iff there is $z\in A$ such that $a=bz^n$;
	\item $A_V^s$ is a dense subset of $1+p^n\mathcal{O}_K$;
	\item $A$ satisfies the dense axioms of $D_n(\xi)$;
	\item $A$ satisfies the same Mann axioms as $G$;
	\item $A$ is a small subset of $K$;
	\item $A\vDash Diag(G/\Q)$.
\end{itemize}

\begin{theorem}\label{full completness} $T_G$ is complete.
\end{theorem}
\begin{proof} With the same notations as in Theorem \ref{completness}, it remains to deal with the case 1. (b) where $\alpha \in A_V$. For as discussed before it is sufficient to prove that the type 
$$tp(\alpha^s/K'(\lambda(A)))\cup tp(\alpha^s/A_V^s)\cup\{(\alpha^s)^k\in a^{-1}D_n(\xi):\ k,n\in \Z, \xi\in \{0,\cdots, p-1\}, a\in {A'^s}_V\}$$ is realised in $L$ where the first type is in the language $\mathcal{L}_{Mac}$ and the second in the language of $p$-valued groups. For by Claim \ref{sub-pseudo-Cauchy}, it is sufficient to prove that it is finitely consistent. By the axioms of density $D_n(\xi)$ this is the case. Let $\beta^s$ be one of its realisation. 
We conclude as in Theorem \ref{completness} case 1.(a) setting $K''=K'(\alpha)^h$, $A''=A\cap K''$, $L''=L'(\beta)^h$, $B''=B\cap L''$.
\end{proof}

\begin{Remark}
If $[n]G$ is not finite for all $n$, then the density axioms for $D_n$ are not sufficient. Nothing can be said about $x_i$'s in $D^s$ whose residue modulo $s+\sum n_j$ is not in $G$. So the type considered in the last proof may not be finitely consistent. 
 One can can still describe the theory of $Th(\Qp, G)$ if $G_{Tor}=\mathbb{F}_s$ and $res\ D=\mathbb{F}_s$. Indeed, in that case $A$ is the direct product between $\lambda(A)$, $G_{Tor}$ and $A_V^s$. So, there is $x\in D$ such that $(x^s)^n=(\alpha^s)^ka$ with $res\ x=\xi$ iff there is $y\in D^s$ $y^n=(\alpha^s)^ka$. So, $A'\langle\alpha\rangle\cong B'\langle\beta\rangle$ iff $A'^s\langle\alpha^s\rangle\cong B'^s\langle\beta^s\rangle$. So, the type of $\alpha$ is determined by $tp(\alpha^s/K'(\lambda(A)))\cup tp(\alpha^s/A_V^s)$.
\par In that case, $Th(\Qp, G)$ is axiomatised by the above axiomatisation where we replace $A$ satisfies the dense axioms of $D_n(\xi)$ by $res\ D\cong G_{Tor}\cong \mathbb{F}_s$.
\end{Remark}

\section{Definable sets in $(K,A)$}\label{definable sets}

In this section and for the rest of the paper, we assume $[n]G$ finite for all $n$.
\par Let $G<\Qp^*$ be a small subgroup with Mann property. Then, $G=\lambda(G)\times D$ where $D=G\cap \Zp^\times$. Let us remark that the structures $(\Qp, +,-,\cdot, (P_n)_{n\in \N}, G, \lambda, (\equiv_n)_{n\in \N},\gamma_g)$ and $(\Qp, +,-,\cdot, (P_n)_{n\in \N}, \lambda(G), D, \lambda, (\equiv_n)_{n\in \N},(\gamma_g)_{g\in G})$ are interdefinable. By Theorem \ref{discrete case}, $T_d(\lambda(G))$ admits the elimination of quantifier. It remains to describe the definable sets of the expansion of $(\Qp, \lambda(G))$ by a predicate for $D$. By Lemma \ref{subgroup Qp}, there is $s,n$ such that $D^s$ is a dense subgroup of $1+p^n\Zp$. We work in the structure $(\Qp, \lambda(G), D)$. The description of definable sets is similar to the one in Theorem 7.5 in \cite{Gunaydin-vdD}. We define the language $\mathcal{L}_{D,\lambda}=\mathcal{L}_{Mac}\cup\{\lambda G,G_V,\lambda, (\equiv_n)_{n\in \N},(\gamma_g)_{g\in G}\}$ where $\lambda G$ is interpreted by $\lambda(G)$ and $G_V$ by $D$.
\par Let $\mathcal{L}_{pV}$ be the language of $p$-valued groups as in Section \ref{valued groups}. Let $\mathcal{L}_{pV}(\Sigma)=\mathcal{L}_{pV}\cup\{\Sigma_{\overline{k}}, \overline{k}\subset \Z\}$ the expansion of the language of $p$-valued groups by predicates $\Sigma_{\overline{k}}$ interpreted in $D^s$ by
$$D^s\vDash \Sigma_{\overline{k}}(\overline{g}) \mbox{ iff } \sum_i k_ig_i=0.$$
We denote by $D(\Sigma)$ the $\mathcal{L}_{pV}(\Sigma)$-structure with underlying set $D^s$.
If $\Phi$ is a $\mathcal{L}_{pV}(\Sigma)$-formula, we define $\Phi_G$ a $\mathcal{L}_{D,\lambda}$-formula defined as follows (by induction on the complexity of $\Phi$):
\begin{itemize}
	\item If $\Phi$ is atomic, $\Phi_G\equiv \Phi$ (where the language $\mathcal{L}_{pV}(\Sigma)$ is interpreted in $\mathcal{L}_{D,\lambda}$);
	\item If $\Phi\equiv \neg \Psi$, $\Phi_G\equiv \neg \Psi_G$;
	\item If $\Phi\equiv \Psi\wedge \theta$,  $\Phi_G\equiv \Psi_G\wedge \theta_G$;
	\item If $\Phi\equiv \exists x \Psi(x)$, $\Phi_G\equiv \exists x\in D^s \Phi_G(x)$.
\end{itemize}
A special formula is a formula of the type
$$\exists \overline{y} \wedge_i y_i\in G_V\wedge \Phi_G(\overline{y}^s)\wedge \theta(\overline{x},\overline{y})$$
where $\Phi$ is a $\mathcal{L}_{pV}(\Sigma)$-formula and $\theta$ is a $\mathcal{L}_{Mac}\cup\{\lambda(A),\lambda\}$-formula.
\begin{lemma}\label{Special formula} Any $\mathcal{L}_{D,\lambda}$-formula is equivalent in $T_G$ to a boolean combination of special formulas.
\end{lemma}

\begin{proof}
Let $(K,A)$, $(L, B)$ be two saturated models of $T_G$ of same cardinality. Let $(K,\lambda(A), D_A)$ and $(L, \lambda(B), D_B)$ be the corresponding $\mathcal{L}_{D,\lambda}$-structure. So, $A=\lambda(A)\times D_A$ and $B=\lambda(B)\times D_B$. Let $\overline{\alpha}\in K^n$ and $\overline{\beta}\in L^n$ that satisfy the same special formulas. We have to prove that $tp_{(K,\lambda(A),D_A)}(\overline{\alpha})=tp_{(L, \lambda(B), D_B)}(\overline{\beta})$.  For it is sufficient to find $\iota$ isomorphism in the back-and-forth system defined in the proof of Theorem \ref{completness} that sends $\overline{\alpha}$ to $\overline{\beta}$.
\begin{claim}trdeg $\Q(A)(\overline{\alpha})=$ trdeg $\Q(B)(\overline{\beta})$.
\end{claim}
\begin{proof} Without loss of generality, $\alpha_1,\cdots ,\alpha_r$ are algebraically independent over $\Q(A)$ and the above transcendence degree is $r$. Assume that $\beta_1,\cdots, \beta_r$ are algebraically dependent over $\Q(B)$. So there is $\phi(\overline{x},\overline{y})$ a $\mathcal{L}_{Mac}$-formula and $\overline{h}\subset H$ such that 
$$(L,B)\vDash \Phi(\overline{h},\beta_1,\cdots, \beta_r)\wedge \exists^{\leq n}x_r\Phi(\overline{h},\beta_1,\cdots,\beta_{r-1},x_r).$$
Assume that $h_i=s_i\lambda(h_i)$ where $s_i\in D_B$. Then,
$$(L,B)\vDash \exists \overline{s}\in D_B \exists \overline{t}\in \lambda(B) \Phi(\overline{st},\beta_1,\cdots, \beta_r)\wedge \exists^{\leq n}x_r\Phi(\overline{st},\beta_1,\cdots,\beta_{r-1},x_r).$$
As $\overline{\alpha}$ and $\overline{\beta}$ satisfy the same special formulas, there is $\overline{g}\subset D_A$ and $\overline{t}\subset \lambda(A)$ such that 
$$(K,A)\vDash \Phi(\overline{gt},\alpha_1,\cdots, \alpha_r)\wedge \exists^{\leq n}x_r\Phi(\overline{g\lambda(h)},\alpha_1,\cdots,\alpha_{r-1},x_r).$$
This is a contradiction with the algebraic independence of $\alpha_1,\cdots ,\alpha_r$ over $A$.
\end{proof}
Let $D'_A:=\{g_1,g_2,\cdots\}$ be a subgroup of $D_A^s$ of cardinality $|G|$ such that $D'_A(\Sigma)\prec D_A^s(\Sigma)$ (as valued groups) and $\overline{\alpha}$ has transcendence degree $r$ over $\Q(D'_A, \lambda(A))$. Let $\widetilde{D}_A=\{g\in A\mid\ g^s\in D'_A\}$. Let $\widetilde{A}_\lambda$ be the image under $\lambda$ of the $\lambda$ closure of $\Q(\widetilde{D}_A,\overline{\alpha})$ (the $\lambda$-closure is of a substructure $M$ is the smallest field that contains $M$ and is closed under $\lambda$ - see construction in the proof of Claim \ref{sub-pseudo-Cauchy}). Let $\theta_1,\cdots, \theta_n$ be $\mathcal{L}_{pV}$-formulas and $\Phi_1,\cdots, \Phi_m$ be $\mathcal{L}_{Mac}\cup\{\lambda(A), \lambda\}$-formulas such that 
$$D_A(\Sigma)^s\vDash\theta_i(\overline{g}^s)\qquad K\vDash \Phi_j(\overline{\alpha},\overline{g})$$
for some $\overline{g}\subset \widetilde{D}_A$ and for all $i,j$. Then,
$$(K,A)\vDash \exists \overline{y}\in D_A  \bigwedge_i\theta_{i,G}(\overline{y}^s)\wedge\bigwedge_j \Phi_j(\overline{\alpha},\overline{y}).$$
As $\overline{\alpha}$ and $\overline{\beta}$ satisfy the same special formulas, 
$$(L,H)\vDash \exists \overline{y}\in D_B \bigwedge_i\theta_{i,G}(\overline{y}^s)\wedge\bigwedge_j \Phi_j(\overline{\beta},\overline{y}).$$
So the following type is consistent:
$$\{ G_V(\overline{y})\}\cup \{\theta_{i,G}(\overline{y}^s)\}\cup \{\Phi_j(\overline{\beta},\overline{y})\},$$
where $\theta_i$, $\Phi_j$ runs over the formulas like above. Let $\widetilde{D}_B=\{h_1,h_2,\cdots\}$ be a realisation of the type in $(L,B)$. Let $D'_B=\widetilde{D}_B^s$ and $\widetilde{B}_\lambda$  be the image under $\lambda$ of the $\lambda$ closure of $\Q(\widetilde{D}_B,\overline{\beta})$. Then $D'_B(\Sigma)\prec D_B^s(\Sigma)$ and the application $g_i^s\rightarrow h_i^s$ induces an isomorphism of valued groups between $D'_A$ and $D'_B$. Let us remark that if there is $y\in\widetilde{D}_A$ with $res\ y=\xi$ then there is $z\in \widetilde{D}_B$ such that $z^s$ is the image of $y^s$ and $res\ z=\xi$ (for it is expressible by a special formula). This $z$ is uniquely determined by the pair $(z^s,res\ y)$. Furthermore, $\overline{\beta}$ has transcendence degree $r$ over $\Q(\widetilde{D}_B, \lambda(B))$. Finally, the extension of this morphism that sends $\overline{\alpha}$ to $\overline{\beta}$ induces a isomorphism of $\mathcal{L}_{Mac}\cup\{\lambda(A), \lambda\}$-structures between $\Q(\widetilde{D}_B,\widetilde{A}_\lambda)(\overline{\alpha})$ and $\Q(\widetilde{D}_B, \widetilde{B}_\lambda)(\overline{\beta})$. Let $K'=\Q(\widetilde{D}_A,\widetilde{A}_\lambda)(\overline{\alpha})^h$, $L'=\Q(\widetilde{D}_B,\widetilde{B}_\lambda)(\overline{\beta})^h$. So, by Fact \ref{alg indep} and construction of $\widetilde{D}_A$ and $\widetilde{D}_B$, $K'\cap D_A= \widetilde{D}_A$ and $L'\cap D_B= \widetilde{D}_B$.
 So, $(K',G')$ and $(L',H')$ are isomorphic $\mathcal{L}_G$-structures. Furthermore, by Theorem \ref{discrete case}, this isomorphism extends to an isomorphism between $K'(\lambda(A))^h$ and $L'(\lambda(B))^h$. So, it belongs to the back-and-forth system.
\end{proof}

\begin{prop}\label{Definable set} 	Let $G= T^\Z\times D<\Qp^*$ be a small subgroup with Mann property and $[n]D^s$ is finite for all $n$. Let $(K,A)$ be a model of $T_G$. Then, every definable subset of $(K,A)$ is a boolean combination of subsets of $K^n$ defined by formulas $\exists \overline{y} A_V(\overline{y})\Phi(\overline{x},\overline{y})$ where $\Phi(\overline{x},\overline{y})$ is a quantifier-free formulas in the language of $p$-adically closed fields extended by $\{\lambda(A), \lambda, T\}$ and $A_V=A\cap \mathcal{O}_K^\times$.
\end{prop}
\begin{proof}
By the above lemma, every formula is equivalent to a boolean combination of special formulas in $\overline{x}$ i.e. formula of the type
$$\Psi(\overline{x})\equiv \exists \overline{y} A_V(\overline{y})\wedge  \theta_{G}(\overline{y}^s)  \wedge \phi(\overline{x},\overline{y}).$$
Let $A_{V,s}=(A\cap\mathcal{O}_K^\times)^{s}$ ($s$th powers).
By quantifier elimination for abelian $p$-valued groups, $\{\overline{g}\in A_{V,s}^m\mid A_{V,s}\vDash \theta(\overline{y})\}$ is equivalent to a boolean combination of
\begin{enumerate}[(a)]
	\item $$\{\overline{g}\in A_{V,s}^m\mid \xi_{\overline{k}}(\overline{g})=1\};$$
	\item $$\{\overline{g}\in A_{V,s}^m\mid V(\xi_{\overline{k}}(\overline{g}))+r\Square{}{}V(\xi_{\overline{k'}}(\overline{g}))+r'\};$$
	\item $$\{\overline{g}\in A_{V,s}^m\mid \xi_{\overline{k}}(\overline{g})\equiv_n i\},$$ 
\end{enumerate} 
where $\xi_{\overline{k}}(\overline{g})=g_1^{k_1}\cdots g_n^{k_n}$ with $k_i\in\Z$.
Any of these formulas is definable by a quantifier-free $\mathcal{L}_{Mac}$-formula and existential quantifier over $A_{V,s}$ (we need that $[n](G\cap \Zp)$ is finite for congruences). So, $\Psi$ is equivalent to
$$\Psi(\overline{x})\equiv \exists \overline{z} A_V(\overline{z})\wedge \phi'(\overline{x},\overline{z}),$$
with $\phi'$ a $\mathcal{L}(\lambda(A),\lambda)$-formula (it can be assumed quantifier-free by Theorem \ref{discrete case}).
\end{proof}

Finally, we describe which subsets of $A$ are definable in the dense case.
\begin{prop}\label{definable set in G} Let $G<\Qp^*$ be a dense subgroup of $1+p^n\Zp$. Let $(K,A)\vDash Th(\Qp,G)$. Let $X$ be a subset of $A^n$ definable in $(K,A)$ (with parameters in $K$). Then there is $Y$ definable in $K$ and $Z$ definable in $A$ (in the language of $p$-valued groups) such that $X=Y\cap Z$. 
\end{prop}
\begin{proof}
It is sufficient to prove the following:
\par Let $(K',A')$, $(L,H)$ be two $|K|$-saturated expansions of $(K,A)$. Let $\overline{g}\in {A'}^n$ and $\overline{h}\in H^n$ such that $tp_f(\overline{g}/K)=tp_f(\overline{h}/K)$ (type in the language of $p$-adically closed fields) and $tp_g(\overline{g}/A)=tp_g(\overline{h}/A)$ (type in the language of $p$-valued groups) then $(K',A',\overline{g})\equiv_K (L,H,\overline{h})$.
\par For it is sufficient to find an element of the back-and-forth system in the proof of Theorem \ref{completness} whose domain and image contain $\overline{g}$ and $\overline{h}$ resp. Take $K_1=K(\overline{g})^h$, $A_1=K_1\cap A'$ and $K_2=K(\overline{h})^h$ and $A_2=K_2\cap H$. By choice of $\overline{g},\overline{h}$ and the proof of Theorem \ref{completness} (case 1. (b)), the application that sends $\overline{g}$ to $\overline{h}$ induces an isomorphism of $\mathcal{L}(G)$-structures between $(K_1,A_1)$ and $(K_2, A_2)$. Furthermore, this latter belongs to the back-and-forth system (remark that as $\lambda(G)$ is trivial, the steps that involve it can be removed from the proof of \ref{completness}). This completes the proof of the proposition.
\end{proof}

\section{Expansion for the subanalytic structure}\label{Expansion for subanalytic structure}
\par  In this section, we consider $\Qp^{an}$ the expansion of $\Qp$ by all restricted analytic functions (in the sense of \cite{Denef-vdD}). Let $\Zp\{\overline{X}\}$ be the set of restricted power series i.e. formal power series $f(\overline{X})=\sum a_I\overline{X}^I$ with $a_I\in\Zp$ and $v_p(a_I)\rightarrow \infty$ as $|I|\rightarrow \infty$. Let 
$$\mathcal{L}_{an}=\mathcal{L}_{Mac}\cup\{ (f)_{f\in \Zp\{X_1,\cdots, X_m\},m\in \N}\}$$
 be the expansion of the Macintyre language for $p$-adically closed fields by function predicates for these series where one interprets $f\in \Zp\{X_1,\cdots, X_n\}$ in $\Qp$ by
$$f(\overline{x})=\left\{\begin{array}{ll}
	\sum a_I\overline{x}^I& \mbox{if }\overline{x}\in \Zp^m\\
	0 & \mbox{ otherwise.}
\end{array}\right.
$$
\par Let $G$ be a subgroup of $\Qp^*$. First, we remark that the case where $G$ is a dense subgroup of $1+p^n\Zp$ is beyond the scope of this paper: consider for instance the expansion of $\Qp^{an}$ by a predicate for the multiplicative group $(1+p)^\Z$. In this expansion, the ring $(\Z,+,\cdot, 0,1)$ is definable. Indeed, let $exp_p(X)=\sum \frac{X^n}{n!}$ and $log_p(1+X)=\sum\frac{(-1)^{n+1}X^n}{n}$ be the exponential and logarithm maps defined by the usual power series. Then, for all $n\in \Z$, $(1+p)^{n}=exp(log(1+p)n)$ (recall that $log(1+X)$ and $exp_p(X)$ converges on $p\Zp$ and therefore the last equality is well-defined as $v_p(\log(1+p))=1$). Therefore the map $exp_p(\log(1+p)X):\Zp\rightarrow 1+p\Zp$ establishes an isomorphism of groups between $(\Z,+)\subset \Qp$ and $((1+p)^\Z,\cdot)$ that is $\mathcal{L}_{an}$-definable - as $exp_p(\log(1+p)X)\in\Zp\{X\}$.
Therefore the ring $(\Z,+,\cdot,0,1)$ is $\mathcal{L}_{an}$-definable. Similarly, for any subgroup of $(1+p\Zp,\cdot)$, the above exponential map induces a structure of ring (isomorphic to a subring of $(\Zp,+,\cdot, 0,1))$. We do not push further this case and focus on discrete subgroups of $\Qp^*$.
\par Let $G$ be a discrete subgroup of $\Qp^*$ i.e. $G=T^\Z\times \mu_s$ for some $T\in p\Zp$ and $\mu_s<\mu_{p-1}$. We will treat the case $G:=\mu_{p-1}\cdot p^\Z$, the kernel of the Iwasawa Logarithm. The general case is similar.

\par Let $\mathcal{L}_{an}(G)=\mathcal{L}_{an}\cup\{A, \lambda, \xi\}$ where $\xi$ is interpreted by a given $(p-1)$th root of unity in $\Qp$, $\lambda$ is a unary function defined in $\Qp$ by $\lambda(x)=p^{v_p(x)}$ for all $x\in\Qp$ and $A$ is a unary predicate interpreted by $G$. Let us recall that for all $x\in \Qp^*$, there is $y_1,\cdots, y_p\in A$ with $y_{i+1}=\xi y_i$ and $v(x)=v(y_i)$ for all $i$ i.e. $x=\lambda(x)\xi^i$ for some $0\leq i<p-1$. This is a first-order statement. Therefore, if $(M,A)$ is a model of $Th(\Qp,G)$ in our language, we have a function $\lambda:M^*\rightarrow A$ such that for all $x,y\in M^*$, $v(x)=v(\lambda(x))$, $\lambda(xy)=\lambda(x)\lambda(y)$, $\lambda(p)=p$, $\lambda(\xi)=1$ and for all $a\in A$, $a=\lambda(a)\xi^i$ for some $0\leq i\leq p-1$. Then, $\lambda{M^*}$ is a subgroup of $A$ which contains $p^\Z$ and $A= \lambda{M^*}\times \mu_s$. Let $\mathcal{L}_{an}^D= \mathcal{L}_{an}\cup\{D\}$ and $\mathcal{L}_{an}^D(G):=\mathcal{L}_{an}(G)\cup \{ D\}$ where $D$ is interpreted in $\Qp$ by
$$D(x,y)=\left\{\begin{array}{ll}
	xy^{-1}& \mbox{if } v(x)\geq v(y)\\
	0 & \mbox{ otherwise.}
\end{array}\right.
$$
Let $T_d^{an}(G)$ be the theory of $\Qp$ in the language $\mathcal{L}_{an}^D(G)$. Then,
\begin{theorem}\label{EQ Log}   The theory $T_d^{an}(G)$ admits the elimination of quantifiers in the language $\mathcal{L}_{an}^D(G)$.
\end{theorem}
\begin{Remark} Theorem \ref{discrete case} claims that $Th(\Qp, +,-,\cdot, 0, 1 , (P_n), \mu_{p-1}p^\Z,\lambda)$ admits the elimination of quantifiers. On the other hand, quantifier elimination for $p$-adic analytic structure in $\mathcal{L}_{an}^D$ is a classical result by Denef and van den Dries \cite{Denef-vdD}.
\end{Remark}
\begin{proof}
Let $\mathcal{M}$ be a $\mathcal{L}_{an}^D(G)$-structure and $\mathcal{N}$ a model of our theory which contains $M$. Let $\mathcal{M}^*$ be a saturated model expansion of $M$. We have to prove that $\mathcal{N}$ embed in $\mathcal{M}^*$ over $M$. First, we prove that the substructure generated by $M$ and $\lambda(N)$ embeds in $M^*$. Let $b\in \lambda(N)\setminus \lambda(M)$. We will denote by $M\langle b\rangle$ the structure generated by $M$ and $b$ (i.e. the set of $f(b,\overline{a})$ for all $\overline{a}\in M^n$ and for all $\mathcal{L}_{an}^D(G)$-term $f$). Then, it is sufficient to prove that $M\langle b\rangle$ embed in $M^*$ over $M$. Let $M(b)$ be the field generated by $M$ and $b$. Note that $M\langle b\rangle$ is an immediate extension of $M(b)$. Indeed, the closure of $M(b)$ under $\mathcal{L}_{an}$-terms is an immediate extension (use Weierstrass preparation theorem like in \cite{Denef-vdD} to prove that for all $\overline{y}\in M^n$, for all $f\in \Zp\{X,\overline{Y}\}$, there is $P(X)\in M[X]$ such that $v(f(b,\overline{y}))=v(P(b))$ - see also later in the proof). Then the closure under $\lambda$ and $D$ is also immediate.

\par As $b\notin \lambda(M)$, $v(M(b))\not=v(M)$. Then, $v(M(b))=vM\oplus v(b)\Z$. Without loss of generality we may assume that $v(b)>0$.

 Now, there exists $\eta\in v(M^*)$ which realizes the same type as $v(b)$ over $v(M)$ and furthermore for all $z\in M^*$ with $v(z)=\eta$ the map $b\longmapsto z$ induces an embedding of valued fields of $M(b)$ into $M^*$ (see Proposition 4.10B in \cite{Prestel}). Let $b'$ be the element of $A(M^*)$ such that $b'=\lambda(z)$ for any $z$ with $v(z)=\eta$. Then, we will prove that the map $\sigma:f(b,\overline{a})\longmapsto f(b',\overline{a})$ for all $\overline{a}\in M^k$ and all $\mathcal{L}_{an}^D(G)$-term $f$ induces an embedding of $\mathcal{L}_{an}^D(G)$-structures of $M\langle b\rangle$ into $M^*$. 

\par First, let us remark that for all $f=\sum_i A_i(\overline{X})Y^i\in \Zp\{Y,\overline{X}\}$ and for all $\overline{a}\in M^n$, either $f(Y,\overline{a})\equiv 0$ or $v(f(b,\overline{a}))=\min_i \{v(A_i(\overline{a})b^i)\}$. For by \cite{Denef-vdD} section 1.4, there is $d\in \N$ and $B_{ij}\in p\Zp\{\overline{X}\}$ such that $\|B_{ij}\|\rightarrow 0$ as $i\rightarrow \infty$ and for all $i\geq d$
$$A_i(\overline{X})=\sum_{j< d} A_j(\overline{X})B_{ij}(\overline{X})\quad (*).$$
If  $A_i(\overline{a})=0$ for all $i<d$, then $f(Y,\overline{a})\equiv 0$. Otherwise, there is $i<d$ such that $v(A_i(\overline{a}))$ is minimal among the $v(A_j(\overline{a}))$ and $i$ is the largest index less than $d$ with this property. This can be written as a first-order sentence $\mu_i(\overline{X})$. Let $\overline{x}\in \Zp^n$ such that $\Zp\vDash \mu_i(\overline{x})$. Let $y\in \Zp$ such that $v_p(y)>0$ and $v_p(A_k(\overline{x})y^k)\not=v_p(A_l(\overline{x})y^l)$ for all $k\not=l<d$ (such that $A_k(\overline{x}),A_l(\overline{x})$ are nonzero). Let $k<d$ such that $v_p(A_k(\overline{x})y^k)$ is minimal. Note that this can be expressed by a first-order formula, $\nu_k(\overline{x},y)$. Then by $(*)$, for all $i\geq d$, $v_p(A_i(\overline{x})y^i)>v_p(A_k(\overline{x})y^k)$. So, if $\theta_{i,k}=\nu_k\wedge \mu_i$,
$$\Zp\vDash \forall \overline{x} \forall y \theta_{i,k}(\overline{x},y)\rightarrow v(f(y,\overline{x}))=v(A_k(\overline{x})y^k).$$
As $v(M(b))=vM\oplus v(b)\Z$, for all $P(X)=\sum p_iX^i\in M[X]$ nonzero, $v(P(b))=\min_i\{v(p_ib^i)\}$. For, if $v(p_ib^i)=v(p_jb^j)$ then $v(p_i)=v(p_j)=\infty$ or $v(b)=\frac{v(p_i)-v(p_j)}{i-j}$ which is a contradiction with $v(b)\notin vM$. Therefore,  for all $\overline{a}\in M^n$ such that $f(Y,\overline{a})$ is nonzero $N\vDash \theta_{i,k}(\overline{a},b)$ for some $i,k$.

\par As we have seen above, for all $f\in\Zp\{Y, \overline{X}\}$ and $\overline{a}\in M^n$ either $f(Y, \overline{a})\equiv 0$ or $f(b,\overline{a})\not=0$ and $f(b',\overline{a})\not=0$. So, $f(b,\overline{a})=0$ iff $\sigma(f(b,\overline{a}))=0$. 
By \cite{Denef-vdD}, we know also that $f(b,\overline{a})\in P_n$ iff $f(b,\overline{a})\in P_n$ (for by the proof of Theorem 1.1 in \cite{Denef-vdD}, $f(b,\overline{a})\in P_n$ iff $g(b,\overline{a})\in P_n$ where $g$ is polynomial in $b$; then use the fact that $\sigma$ is an embedding of analytic valued fields). We will now prove that $f(b,\overline{a})\in A$ iff $f(b,\overline{a})\in A$. For as we have seen above, if $f(b,\overline{a})\not=0$, then $v(f(b,\overline{a}))=v(A_k(\overline{a})b^k)$ for some $k<d$. Then by definition of $\lambda$, $\lambda(f(b,\overline{a}))=\lambda(A_k(\overline{a}))b^k$. Let us remark that for all $t\in N$, $t\in A$ iff $t=\xi^i\lambda(t)$ for some $0\leq i<p$. So, $f(b,\overline{a})\in A$ iff $f(b,\overline{a})=\xi^i\lambda(f(b,\overline{a}))$  iff $f(b,\overline{a})-\xi^i\lambda(A_k(\overline{a}))b^k=0$ (note that this is a condition of the type $g(b,\overline{a}, \lambda(A_k(\overline{a})))=0$ for some $g\in \Zp\{Y, \overline{X},Z\}$) iff $f(b',\overline{a})-\xi^i\lambda(A_k(\overline{a}))b'^k=0$ iff $f(b',\overline{a})\in A$.
\par Let us now prove the same properties hold for all $\mathcal{L}_{an}^D$-terms. This follows from the above arguments and \cite{vdD-H-M}. A \emph{$M$-affinoid set} is a set of the form
$$\{x\in N\mid\ v(x-a_1)>v(\gamma_1)\wedge \bigwedge_{1<i\leq m} v(x-a_i)\leq v(\gamma_i)\},$$ 
where the balls of centre $a_i$ and radius $\gamma_i$ are two-by-two disjoint and contained in the ball of centre $a_1$ and radius $\gamma_1$. Furthermore, $a_j,\gamma_k\in M$ for all $1\leq j,k\leq m$.
 Let $t(T,\overline{X})$ be a $\mathcal{L}_{an}^D$-term. By \cite{vdD-H-M} Proposition 4.1, \cite{vdD-H-M} Corollary 3.4 and by choice of $b$, there is an affinoid set $F$ that contains $b$, a unit $E(\overline{X},\overline{Z})\in \Zp\{\overline{X},\overline{Z}\}$  and a rational function $r(Y)\in M(Y)$ such that
 $$t(b,\overline{a})=E(\overline{a},\Psi(b))r(b)$$
 where $\Psi(b)=(\frac{b-a_1}{\gamma_1},\frac{\gamma_2}{b-a_2},\cdots, \frac{\gamma_m}{b-a_m})$. So, $v(t(b,\overline{a}))=v(r(b))=v(r_ib^i)$ for some $r_i\in M$ and $i\in \Z$ and therefore $\lambda(t(b,\overline{a}))=\lambda(r_i)b^i$. One argue like above to prove that $t(b,\overline{a})\in A$ iff $t(b',\overline{a})\in A$. By Lemma 4.2 in \cite{vdD-H-M}, $t(b,\overline{a})\in P_n$ iff $t(b',\overline{a})\in P_n$.

\par Finally, it remains to observe that for all $\mathcal{L}_{an}^D(G)$-term $h(T,\overline{X})$, there is a $\mathcal{L}_{an}^D$-term $t(T,\overline{X},\overline{X'})$ such that $t(b,\overline{a})=h(b,\overline{a},\overline{a}')$ for some $\overline{a}'\in M^k$ (that depends only on $\overline{a})$. For argue by induction on the complexity of the term and use that as we have seen above $\lambda(t'(b,\overline{a}))=\lambda(h'(\overline{a}))b^i$ for all $\mathcal{L}_{an}^D$-term $t'$. Then, it follows from above that $\sigma$ is an embedding of $\mathcal{L}_{an}^D(G)$-structures. Therefore, $M\langle b\rangle$ embeds in $M^*$ and by transfinite induction $M\langle \lambda(N)\rangle$ embeds in $M^*$.
\par Let $M'=M\langle \lambda(N)\rangle$. Then, $N$ and $M^*$ contains $M'$. Let $x\in N\setminus M'$. It remains to prove that $M'\langle x\rangle$ embeds in $M^*$ over $M'$ to conclude to proof of quantifier elimination. Let us remark that $v(M'(x))=v(M')$ (by definition of $\lambda$). Then, for any term $t$ with parameter in $M'$, $\lambda(t(x))\in \lambda(M')$ . So, it can be replaced by a parameter in $M'$. In other words, it is sufficient to find an embedding of $\mathcal{L}_{an}^D$-structures. Therefore, this part follows from the quantifier elimination result in \cite{Denef-vdD}. This concludes the proof of the theorem.
\end{proof}

\section{Application: Model theory of the $p$-adic Iwasawa Logarithm}\label{Iwasawa Logarithm}
Let us remind that the $p$-adic Iwasawa logarithm is defined by $LOG(x):= \log_p(y)$ if $x=p^n\xi y\in \Qp$ with $\xi\in \mu_{p-1}$ and $y\in 1+p\Zp$.
Let us remark that as $\log_p(1+px)\in \Zp\{X\}$, $LOG$ is definable in the structure $(\Qp^{an}, \mu_{p-1}p^{\Z})$: for all $x\in \Qp, x=y(1+pz)$ for some (unique) $z\in \Zp$ and $y\in \mu_{p-1}p^{\Z}$, then $LOG(x)=log_p(1+pz)$.
In this section, we will define a sublanguage of $\mathcal{L}_{an}^D(G)$ for which the theory of $\Qp$ is model-complete following the strategy of \cite{Mariaule2}. First, we define the Weierstrass system generated by $\log(1+px)$: let $F\subset \Zp\{\overline{X}\}$ (in our case, we are interested by $F=\{\log(1+px)\}$). We define by induction on $m$ a family of rings $W^{(m)}_{F,n}\subset \Zp\{X_1,\cdots, X_n\}$. Let $W^{(0)}_{F,n}$ be the ring generated by $\Z[X_1,\cdots , X_n]$ and the elements of $F$. Assume that $W^{(k)}_{F,n}$ is defined for all $n\in \N$ and all $k\leq m$. Then $W^{(m+1)}_{F,n}$ is the ring generated by:
\begin{enumerate}
	\item $W^{(m)}_{F,n}$;
	\item $f(X_{\sigma(1)},\cdots , X_{\sigma(n)})$ for all $f\in W^{(m)}_{F,n}$ and $\sigma$ permutation of $\{1,\cdots , n\}$;
	\item $f^{-1}$ for all $f\in W^{(m)}_{F,n}$ invertible in $\Zp\{X_1,\cdots, X_n\}$;
	\item $f/g(0)$ for all $f,g\in W^{(m)}_{F,n}$ such that $f$ is divisible by $g(0)$ in $\Zp\{\overline{X}\}$;
	\item $A_0,\cdots, A_{d-1}, Q$ obtained by Weierstrass division: for all $d$ and $f,g\in W^{(m)}_{F,n+1}$  with $f$ regular of order $d$ in $X_{n+1}$, by Weierstrass division, there is $A_0, \cdots, A_{d-1}, Q$ such that
	$$g(\overline{X})=f(\overline{X})Q(\overline{X})+[A_0(\overline{X}')+\cdots+A_{d-1}(\overline{X}')X_{n+1}^d], $$
	where $\overline{X}=(X_1,\cdots , X_{n+1}), \overline{X}'=(X_1,\cdots, X_n)$. Then, we require $Q\in W^{(m+1)}_{F,n+1}$ and $A_0, \cdots, A_{d-1}\in  W^{(m+1)}_{F,n}$.
\end{enumerate}

\begin{definition} The \emph{Weierstrass system generated by $F$} is the collection of $W_{F,n}:=\bigcup_m W^{(m)}_{F,n}$ for all $n\in\N$. We denote this system by $W_F$
\end{definition}

\begin{lemma} $W_F$ is a Weierstrass system in the sense of \cite{C-L} Definition 4.3.5.
\end{lemma}
\begin{proof}
From the definition of $W_F$, it is closed under Weierstrass division and therefore, it is a pre-Weierstrass system in the sense of Definition 4.3.3 \cite{C-L}. It remains to prove condition (c) in Definition 4.3.5: We have to show that for all $f=\sum_I a_I\overline{Y}^I\in W_{F,n}$ then there is a finite set $\mathcal{J}\subset \N^n$ such that for all $J\in \mathcal{J}$, there is $g_J\in W_{F,n}\cap p\Zp\{\overline{X}\}$ such that 
$$f=\sum_{J\in\mathcal{J}} a_J(1+ g_J)\overline{Y}^J\qquad (*).$$
We proceed by induction on $n$: If $n=1$, then $f(X)=a_0+a_1X+a_2X^2+\cdots$. As $f\in\Zp\{X\}$, there is $d\in \N$ such that $v(a_i)\geq v(a_d)$ for all $i<d$ and $v(a_j)>v(a_d)$ for all $j>d$.
Take $\mathcal{J}=\{1,\cdots, d\}$, $g_i=0$ for $i<d$ and $g_d=\sum_{j>d} a_ja_d^{-1}X^{j-d}$. It satisfies all the conditions required: Indeed, by definition $g_i\in p\Zp\{X\}$ for all $i$. As $a_i= \frac{\partial^i}{\partial X} f(0)/i!$, $a_i\in W_F$ for all $i$ by closure under Weierstrass division (so $W_F$ is closed under partial derivative) and condition 4. Then, $g_d=(f-[a_0+a_1X+\cdots a_{d-1}X^{d-1}])/(a_dX^d)\in W_F$ by Weierstrass division (in particular, $W_F$ is closed under composition) and  condition 4.
\par Assume now that $(*)$ is satisfied for all $g\in W_{F,n}$. Let $f=\sum_K a_K\overline{X}^K\in W_{F,n+1}$. As $a_K\rightarrow 0$ when $|K|\rightarrow \infty$, we can find $I=(i_1,\cdots, i_{n+1})$ such that 
\begin{itemize}
	\item $v(a_I)$ is minimal among the $v(a_K)$;
	\item $v(a_J)>v(a_I)$ for all $J=(j_1,\cdots, j_{n+1})$ with $j_l> i_l$ for some $l\leq n+1$ and $j_k\geq i_k$ for all $k\not=l$;
	\item $I$ is maximal for these two properties (with respect to the lexicograpical order).
\end{itemize}
Let $g_I:=\sum_{K\in\N^{n+1}\setminus\{0\}} a_{I+K}a_I^{-1}\overline{X}^{K}\in W_{F,n+1}$. Then $\sum_K a_{I+K}\overline{X}^K=(1+g_I)a_I\overline{X}^I$. Fix $k\leq n+1$ and $s\leq i_k$. By $J_{k,s}$ we denote an element of $\N^{n+1}$ with $k$th coordinate $s$. Then,
\begin{align*}
f'_{k,s}:&=\frac{1}{s!}\frac{\partial^s f}{\partial X_k^s}(X_1,\cdots, X_{k-1},0,X_{k+1},\cdots ,X_{n+1})\\
&=\sum_{J_{k,s}} a_{J_{k,s}}X_1^{j_1}\cdots X_{k-1}^{j_{k-1}}X_{k+1}^{j_{k+1}}\cdots X_{n+1}^{j_{n+1}}\in W_{F,n}.
\end{align*}
Fix some order on the couple $(k,s)$ like above. Then, let
$$f_{k,s}=f'_{k,s}-\sum_{(k',s')<(k,s)} \frac{1}{s!}\frac{\partial^s f_{k',s'}}{\partial X_k^s};$$
(in $f_{k,s}$ we only take coefficients $a_{J_{k,s}}$ in $f'_{k,s}$ that did not appear previously in some other $f_{k',s'}$ in order to avoid repetition of coefficients in next equalities).
By induction hypothesis, there are $\mathcal{J}_{k,s}$ and $g'_L\in W_{F,n}\cap p\Zp\{\overline{X}\}$ for all $L\in \mathcal{J}_{k,s}$ such that 
$$f_{k,s}=\sum_{J\in\mathcal{J}_{k,s}} a_JX_1^{j_1}\cdots X_{k-1}^{j_{k-1}}X_{k+1}^{j_{k+1}}\cdots X_{n+1}^{j_{n+1}}(1+ g'_J).$$
Take $\mathcal{J}:=\bigcup_{k,s} \mathcal{J}_{k,s}\cup\{I\}$ and $g_L:=g'_LX_k^s$ if $L\in \mathcal{J}_{k,s}$ and $g_L=0$ otherwise. Then,
$$f=\sum_{k,s} f_{k,s}X_k^s+\sum_{K\in \N^{n+1}} a_{I+K}\overline{X}^{I+K} = \sum_{J\in\mathcal{J}} a_J\overline{X}^J(1+ g_J). $$
\end{proof}	
From this, it follows from \cite{C-L} that
\begin{theorem}
The theory of $\Qp$ admits the elimination of quantifiers in the language $(+,-,\cdot, 0, 1, (P_n)_{n\in\N}, (f)_{f\in W_F}, D)$.
\end{theorem}
Let $W_{\log} =W_F$ for $F=\{\log_p(1+px)\}$.
A variation of the proof of Theorem \ref{EQ Log} using the above quantifier elimination shows that
\begin{cor}\label{EQ}
$Th(\Qp, +,-,\cdot, 0, 1, (P_n)_{n\in\N}, (f)_{f\in W_{\log}}, D, A, \xi, \lambda)$ admits the elimination of quantifiers.
\end{cor}
In the above corollary, $A$ is a unary predicate interpreted by $\mu_{p-1}p^\Z$. 
\par In the above theorem, it may not be obvious to describe all the elements in $W_{\log}$. We will now give a countable language in which our theory is model-complete. The key idea is that in this language the elements of $W_{\log}$ are existentially definable. First, let us define a family of finite algebraic extensions of $\Qp$: We fix $K_n$ an finite algebraic extension of $\Qp$ such that 
\begin{enumerate}
	\item $K_n$ is the splitting field of $Q_n$ for some $Q_n\in\Q[X]$ of degree $d(n)$;
	\item $K_n=\Qp(\beta_n)$ for any $\beta_n$ root of $Q_n$ and $V_n:=\Zp[\beta_n]$ is the valuation ring of $K_n$;
	\item $K_n$ contains $K_{n-1}$ and all algebraic extension of $\Qp$ of degree less than $n$.
\end{enumerate} 
Let $\mathbf{y}\in V_n$. Then, $\mathbf{y}=y_0+y_1\beta_n+\cdots y_{d(n)-1} \beta_n^{d(n)-1}$ for some $y_i\in\Zp$. Let us remark that $\log_p(1+pX)$ is convergent on $V_n$. And,
\begin{align*}
\log_p(1+p\mathbf{y})&=\log_p\left(1+p\left(\sum y_i\beta_n^i\right)\right)\\
&= c_{0,n}(\overline{y})+c_{1,n}(\overline{y})\beta_n+\cdots c_{d(n)-1,n}(\overline{y})\beta_n^{d(n)-1}, \qquad (\alpha)
\end{align*}
for some $c_{i,n}(\overline{y})\in \Zp$. These coefficients determine functions from $\Zp^{d(n)}$ to $\Zp$. Note that by rearranging formally the power series $\log_p(1+p\sum Y_i\beta_n^i)$, one shows that $c_{i,n}(\overline{Y})\in\Zp\{\overline{X}\}$. If we identify $V_n$ with its structure of $\Zp$-module, we see that:
\begin{lemma} The structure $(V_n,+,-,\cdot, 0,1, \log_p(1+pX))$ is definable in $(\Qp,+,$ $-,\cdot, 0,1, \log_p(1+pX), (c_{i,n})_{i\leq d(n)})$.
\end{lemma}
Even more is true (see \cite{Mariaule2} Proposition 4.2):
\begin{lemma} The structure $(V_n,+,-,\cdot, 0,1, \log_p(1+pX),(c_{i,n})_{i\leq d(n)})$ is definable
 in $(\Qp,+,-,\cdot, 0,1, \log_p(1+pX), (c_{i,n})_{i\leq d(n)})$.
\end{lemma}
It follows from Proposition 5.1 \cite{Mariaule2} that
\begin{lemma} Let $F=\{\log_p(1+px), c_{i,n}(\overline{X}),n\in \N, i\leq d(n)\}$. Then for all $f\in W_F$, $f$ is existentially definable (as function from $\Zp^k$ to $\Zp$) in terms of the elements of $F$ and their derivatives.
\end{lemma}
As the derivatives of $\log_p(1+px)$ are existentially definable, so are the derivatives of $c_{i,n}$ (using the equality $(\alpha)$). Then by the last lemma, the elements of $W_F$ are definable in terms of the elements of $W_F^{(0)}$. We combine this with Corollary \ref{EQ} to get
\begin{prop}\label{model-complete}
$Th(\Qp, +,-,\cdot, 0, 1, A, \xi, \lambda, \log_p(1+px), c_{i,n}(\overline{X}),n\in \N, i\leq d(n))$ is model-complete.
\end{prop}
\begin{cor}
$Th(\Qp, +,-,\cdot, 0, 1, A, \xi, \lambda, LOG, c_{i,n}(\overline{X}),n\in \N, i\leq d(n))$ is model-complete.
\end{cor}

\section{Non-independence property}\label{NIP}
First, we prove that the theory of $\Qp$ in the language $\mathcal{L}_{an}^D(p^\Z)$ is NIP.
Let us recall the following result from \cite{C-S}:
\begin{theorem}[Corollary 2.5 \cite{C-S}]\label{C-S 2.5} Let $T$ be a first-order $\mathcal{L}$-theory. Let $M\vDash T$ and $A\subset M$. Let $T_P$ be the theory of $(M,A)$ where we extend the language by a predicate $P$ interpreted in $M$ by $A$. Assume that $T$ is NIP, $A_{ind}(L)$ is NIP and  $T_P$ is bounded. Then $T_P$ is NIP.
\end{theorem}
In the above theorem, $A_{ind}(L)$ is the structure induced on $A$ by the $\mathcal{L}$-definable sets in $M$. $T_P$ bounded means that any $\mathcal{L}\cup\{P\}$-formula is equivalent to a formula of the form $Q_1 x_1\in P\cdots  Q_n x_n \in P \Psi (\overline{x},\overline{y})$ with $Q_i\in \{\exists, \forall\}$ and $\Psi$ is a $\mathcal{L}$-formula.
\par In the case of $T_d^{an}(p^{\Z})$, by our result of quantifier elimination in Section \ref{Expansion for subanalytic structure}, the theory is bounded (in the language $\mathcal{L}_{Mac}\cup\{A\}$ - as the function $\lambda$ is definable by a bounded formula). Furthermore, it is known that $T$ is NIP \cite{Haskell-Macpherson}. It remains to prove that $A_{ind}$ is NIP. For we will show that the induced structure is exactly the Presburger arithmetic.
\begin{lemma}
$(A,+_A,0_A, 1_A, (\equiv_n)_{n\in\N}, <)$ is $\mathcal{L}_{Mac}$-definable i.e. $A_{ind}$ is an expansion of Presburger arithmetic.
\end{lemma}
\begin{proof}
Let $(M,A, +_M,\cdot_M,0_M,1_M)$ be a model of $T_d^{an}(p^{\Z})$.
The addition in the group $A$ is given by the multiplication in the field and $0_A=1_M$, $1_A=p_M$. The order is defined by $x<y$ iff $v(x)<v(y)$ for all $x,y\in A$. Finally, note that $x\in A\cap P_n$ iff $x$ is $n$ divisible in $A$. 
\end{proof}
We denote by $\mathcal{L}_{Pres}$ the language of Presburger $(+,-,0,1,<,\equiv_n (n\in\N))$.
We will prove now that any definable set in $A_{ind}$ is $\mathcal{L}_{Pres}$-definable.

\begin{lemma}\label{composition with lambda}
Let $f=\sum a_i(\overline{Y})X^i\in\Zp\{\overline{Y},X\}$. Let $(M,A)\vDash Th(\Qp^{an}, p^{\Z})$. Then there exists $D(f)$ such that for all $\overline{y}\in M^k$ such that $f(\overline{y},X)$ is nonzero, for all but finitely many $x\in A^{\geq 0}$, $\lambda(f(\overline{y}),x)=\lambda(a_i(\overline{y}))x^i$ for some $i<D(f)$ (depending on $x$). Furthermore, for fixed $i$, the set of $x\in A^{\geq 0}$ such that $\lambda(f(x))=\lambda(a_i)x^i$ is definable using the order on $A$ (induced by the order on the valuation group).
\end{lemma}
\begin{proof} Let $\overline{y}$ such that $f(\overline{y},X)$ is nonzero. Then there is $i$ such that $v(a_i(\overline{y}))<\infty$.
Let $d=d(\overline{y})$ be the largest index such that $v(a_d(\overline{y}))$ is minimal among the $v(a_i(\overline{y}))$. This number exists and is bounded uniformly over $\overline{y}$ by some $D(f)\in\N$  (see Lemma 1.4 in \cite{Denef-vdD}). Then for all $x\in A^{>0}$ and $i>d$, $v(a_i(\overline{y})x^i)>v(a_d(\overline{y})x^d)$. Let $P(X)=a_0(\overline{y})+\cdots +a_d(\overline{y})X^d$. First, let us remark that if for all $i\not=j$ $v(a_i(\overline{y})x^i)\not= v(a_j(\overline{y})x^j)$ then $v(P(x))=\min_i\{v(a_i(\overline{y})x^i)\}$. In this case, $\lambda(P(x))=\lambda(a_i(\overline{y})x^i)=\lambda(a_i(\overline{y}))\lambda(x)^i=\lambda(a_i(\overline{y}))x^i$ for the index $i$ such that $v(a_i(\overline{y})x^i)$ reaches the minimum. Then, $v(f(x))=v(P(x)+\sum_{k>d}a_k(\overline{y})x^k)=v(P(x))$ (this is first-order and true in $\Zp$ for all parameters $(\overline{y}',x')$ with the same properties, so this holds in $M$). So, $\lambda(f(x))=\lambda(a_i(\overline{y}))x^i$ for some $i\leq d(\overline{y})\leq D(f)$.
\par Assume that $v(a_i(\overline{y})x^i)=v(a_j(\overline{y})x^j)<\infty$ for some $j<i \leq D(f)$. Then, $v(a_i(\overline{y}))-v(a_j(\overline{y}))=(i-j)v(x)$. This equality can be satisfied by a unique $x\in A$ (as all $x\in A$ have distinct valuation). So , for all $x\in A$ except finitely many point, either $a_i(\overline{y})=0$ for all $i<D(f)$ (so, $f(\overline{y},X)\equiv 0$) or there is a unique $i<D(f)$ such that $v(a_i(\overline{y})x^i)$ is minimal. This complete the proof of the first part of the lemma. 
\par Let $i<D(f)$. Let us proof that the set of $x\in A$ such that $v(a_i(\overline{y}))$ is the unique minimum is definable using the order. If $f(\overline{y},X)\equiv 0$, then any $x\in A$ satisfies the condition $\lambda(f(\overline{y},x))=\lambda(a_i(\overline{y})x^i)=\lambda(0):=0$. Otherwise, the minimum is reached. Let us remark that $v(a_i(\overline{y})x^i)$ is the unique minimum among the $v(a_k(\overline{y})x^k)$  iff $(v(a_k(\overline{y}))-v(a_i(\overline{y})))/(k-i)>v(x)$ for all $i<k<D$ and $(v(a_i(\overline{y}))-v(a_k(\overline{y})))/(i-k)<v(x)$ for all $i>k$. So, our condition is equivalent to an interval (possibly with boundary to infinity) in $A$.
\end{proof}
\begin{Remark} Like in the above lemma, if $h$ is a $\mathcal{L}_{an}^D$-term and $\overline{y}\in M^n$ then for all $a\in A$ except finitely many, $\lambda(h(a,\overline{y}))=\lambda(r(\overline{y}')a^i)$ for some $\overline{y}'\in M^{n'}$, $r$ rationnal function and $i\in \Z$. For proceed like in the proof of Theorem \ref{EQ Log}, to prove that $h(a,\overline{y})$ is the product between a unit and a rational function. Then we argue like in the above lemma where $P$ is now a rational function.
\end{Remark}

\begin{theorem} Any definable set in $A_{ind}$ is $\mathcal{L}_{Pres}$-definable.
\end{theorem}
\begin{proof}
First, we show that the definable subsets of $A$ are $\mathcal{L}_{Pres}$-definable. Let $X$ be such a set. By quantifier elimination, $X$ is boolean combination of sets of the type:
$$\{x\in A\mid\ f(x)\in \mu P_n\}\qquad (1)$$
$$\{x\in A\mid f(x)=0\}\qquad (2)$$
$$\{x\in A\mid\ f(x)\in A\}\qquad (3)$$
where $f$ is a term of the language $(+,-,\cdot,0, 1, \xi, \lambda, A, (P_n)_{n\in \N}, D, (g)_{g\in\Zp\{\overline{X}\}})$ (we allow parameters from $M$). By Lemma \ref{composition with lambda} and the remark after, there is a finite union of intervals and points such that on each interval $I$, the set $X\cap I$ is equivalent to a definable set $X'$ where any term that appears in the definition of $X'$ is a term where the function $\lambda$ applies only on the parameters. So, we may assume that $X$ is defined by terms of the type $f(X,\overline{y})$ where $f$ is a $\mathcal{L}_{an}^D$-term and $\overline{y}\in M^n$. We may assume that $f(X,\overline{y})$ is not identically zero (otherwise, the above sets are trivially $\mathcal{L}_{Pres}$-definable).
\par The case (2) is a finite set. In case (3), by Lemma \ref{composition with lambda}, we see that $\lambda(f(x,\overline{y}))=\lambda(a_i(\overline{y}))x^i$ for some $i$ for all but finitely many $x\in A$. Furthermire, $f(x,\overline{y})\in A$ iff $\lambda(f(x,\overline{y}))=f(x,\overline{y})$. Then $f(x,\overline{y})=\lambda(a_i(\overline{y}))x^i$. Then, $f(x,\overline{y})\in A$ iff $x$ is a root of $f(X)-\lambda(a_i(\overline{y}))X^i$. If this later series is not identically zero, it has finitely many roots in $A$ (in particular the set of such $x$ is definable). Otherwise, the series is identically zero and therefore any points $x\in A$ such that $v(a_i(\overline{y})x^i)$ is minimal belongs to (3). This is a $\mathcal{L}_{Pres}$-definable condition by Lemma \ref{composition with lambda}. 
\par Finally, in the case (1), as by Lemma \ref{composition with lambda} the $\lambda$ function occurs only in the parameters, we can assume by analytic cell decomposition \cite{Cluckers2} that the set is of the type
$$\{x\in A\mid v(\alpha)\square_1 v(x-c)\square_2 v(\beta)\quad x-c\in\mu P_n\}, $$
for some $\alpha, \beta,c\in M$, $\mu,n\in \N$. The first condition $v(\alpha)\square_1 v(x-c)\square_2 v(\beta)$ clearly determines a finite union of intervals (possibly with bounds to infinity) in $A$ and points: for all $t\in A$, if $t\not=\lambda(c)$ then $v(t-c)=\min\{v(t),v(c)\}$. So, for all $t$ such that $v(t)>v(c)$, $t$ satisfies the first condition iff $v(\alpha)\square_1 v(c)\square_2 v(\beta)$. If $v(t)<v(c)$ then $t$ satisfies the first condition iff $v(\alpha)\square_1 v(t)\square_2 v(\beta)$ which is an interval. As there is only one $t\in A$ such that $v(t)=v(c)$ this last case is equivalent to $v(\alpha)\square v(\lambda(c)-c)\square v(\beta)$: either $\lambda(c)$ is in $X$ or not. So, the first condition is $\mathcal{L}_{Pres}$-definable.
\par It remains to prove that the condition $x\in A\wedge x-c\in\mu P_n$ is $\mathcal{L}_{Pres}$-definable. Let us recall that for all $z\in \Qp$, $z\in \mu P_n$ iff $n$ divides $v_p(\mu^{-1}z)$ and $\mu^{-1}zp^{-v(\mu^{-1}z)}\in P_n$. So, for all $x\in A$, $x-c\in\mu P_n$ iff $n$ divides $v(\mu^{-1}(x-c))$ and  $\mu^{-1}(x-c)\lambda(\mu^{-1}(x-c))^{-1}\in P_n$. We split this condition into the three possible cases according to whether or not $v(x)$ is larger than $v(c)$. Then we will show that each case is $\mathcal{L}_{Pres}$-definable. The three possibilities are:
\begin{enumerate}[(a)]
	\item $v(x)<v(c)$, $n$ divides $v(\mu^{-1}x)$ and $\mu^{-1}\lambda(\mu)- \mu^{-1}\lambda(\mu)cx^{-1}\in P_n$;
	\par Indeed, in the case where $v(x)<v(c)$, $v(\mu^{-1}(x-c))=v(\mu^{-1}x)$. So, $n$ divides $v(\mu^{-1}(x-c))$ iff $n$ divides $ v(\mu^{-1}x)$. On the other hand, $\lambda(\mu^{-1}(x-c))=\lambda(\mu^{-1})\lambda(x)=\lambda(\mu^{-1})x$. Therefore, in this case $x-c\in\mu P_n$,	 is equivalent to the above relation.
	\item $v(x)>v(c)$, $n$ divides $v(\mu^{-1}c)$ and $\mu^{-1}x\lambda(\mu^{-1}c)^{-1}-\mu^{-1}c\lambda(\mu^{-1}c)^{-1}\in P_n$;
	\item $x=\lambda(c)$ and $c-\lambda(c)\in \mu P_n$.
\end{enumerate}
\par In case (a), $n$ divides $v(\mu^{-1}x)$ holds iff $\lambda(\mu)\equiv_n x$: this is $\mathcal{L}_{Pres}$-definable. Then let us recall that for all $n$, there is $K(n)$ such that for all $a,b\in \Qp$ with $v(b)>K(n)$, $a(1+b)\in P_n$ iff $a\in P_n$. Apply this with $a=\mu^{-1}\lambda(\mu)$ and $b=-cx^{-1}$. Therefore, for all $x\in A$ with $v(x)<v(c)-K(n)$, the second condition of (a) holds iff $\mu^{-1}\lambda(\mu)\in P_n$: this is independent of $x$ and so is $\mathcal{L}_{Pres}$-definable. Finally, for all $x\in A$ with $v(c)-K(n)\leq v(x)<v(c)$ either $x\in X$ or $x\notin X$. As there is only finitely many $x$ that satisfies the condition $v(c)-K(n)\leq v(x)<v(c)$, the set of such $x\in X$ corresponds to a finite union of points.
\par The case (b) is similar: the first condition becomes $n$ divides $v(\mu^{-1}c)$ and so does not depends on $x$. The second condition is equivalent for $v(x)>v(c)+K(n)$ to $\mu^{-1}c\lambda(\mu^{-1}c)^{-1}\in P_n$. And for $v(x)<x\leq v(c)+K(n)$ it is satisfied by finitely many $x\in A$. Finally, case (c) does not depend on $x$.

We regroup all these conditions and we see that $X$ is indeed $\mathcal{L}_{Pres}$-definable.
\par This proves that $Th(A_{ind})$ is $\mathcal{L}_{Pres}$-minimal (in the the sense of \cite{Cluckers}). So by Theorem 6 in \cite{Cluckers}, any definable set of $A^n$ is $\mathcal{L}_{Pres}$-definable.

\end{proof}
We combine this theorem with Theorem \ref{C-S 2.5} to obtain:
\begin{cor}The theory of $\Qp$ in the language $\mathcal{L}_{an}^D(p^{\Z})$ is NIP.
\end{cor}
\begin{cor} The theory of $(\Qp, LOG)$ is NIP.
\end{cor}
\begin{proof} This immediate from the above corollary as $(\Qp, LOG)$ is $\mathcal{L}_{an}^D(p^{\Z})$-definable: for all $x\in \Qp$, there is a unique $\xi\in\mu_{p-1}, a\in p^{\Z}$ and $y\in \Zp$ such that $x=a\xi(1+py)$. This decomposition is $\mathcal{L}_{an}^D(p^{\Z})$-definable. So, $LOG(x)=LOG(a)+LOG(\xi)+LOG(1+py)=\log_p(1+py)$ which is $\mathcal{L}_{an}$-definable.
\end{proof}

To conclude this section, we prove that the dense case is also NIP in the finitely generated case.
\begin{theorem} The theory of $(\Qp, G)$ is NIP for all $G$ finitely generated subgroup of $1+p\Zp$.
\end{theorem}
\begin{proof} By Theorem \ref{C-S 2.5} and Proposition \ref{Definable set}, it is sufficient to prove that the structure induced on the group $G$ is NIP. By Proposition \ref{definable set in G}, formulas in $A_{ind}$ are of the type $\Phi\wedge \Psi$ where $\Phi$ is a formula in the language of $p$-adically closed fields (with parameters) and $\Psi$ is a formula in the language of $p$-valued groups. As indiscernible sequences in $A_{ind}$ are indiscernible in $K$ (as valued fields) and in $A$ (as valued groups), as the theory of $p$-adically closed fields is NIP and by Theorem \ref{p-valued group are NIP}, $\Phi$ and $\Psi$ are NIP. Therefore, $Th(A_{ind})$ is NIP.
\end{proof}

\bibliographystyle{plain}
\bibliography{Biblio_p-adic}

\begin{thebibliography}{10}

\bibitem{Belegradek-Zilber}
O.~Belegradek and B.~Zilber.
\newblock The model theory of the field of reals with a subgroup of the unit
  circle.
\newblock {\em J. London Math. Soc.}, 78(3):563--579, 2008.

\bibitem{C-S}
A.~Chernikov and P.~Simon.
\newblock Externally definable sets and dependent pairs.
\newblock {\em Israel Journal of Mathematics}, 194(1):409--425, 2013.

\bibitem{Cluckers2}
R.~Cluckers.
\newblock Analytic $p$-adic cell decomposition and integrals.
\newblock {\em Transactions of the American Mathematical Society},
  356(4):1489--1499, 2003.

\bibitem{Cluckers}
R.~Cluckers.
\newblock Presburger sets and {P}-minimal fields.
\newblock {\em Journal of Symbolic Logic}, 68(1):153--162, 2003.

\bibitem{C-L}
R.~Cluckers and L.~Lipshitz.
\newblock Fields with analytic structure.
\newblock {\em Journal of the {E}uropean Mathematical Society},
  13(4):1147--1223, 2011.

\bibitem{Denef-vdD}
J.~Denef and L.~van~den Dries.
\newblock $p$-adic and {R}eal {S}ubanalytic {S}ets.
\newblock {\em Annals of Mathematics, Second Series}, 128(1):79--138, 1988.

\bibitem{Evertse}
J.H. Evertse.
\newblock On sums of {$S$}-units and linear recurrences.
\newblock {\em Composition Mathematica}, 53(2):225--244, 1984.

\bibitem{Guignot-thesis}
F.~Guignot.
\newblock {\em Th\'eorie des mod\`eles des groupes ab\'eliens valu\'es}.
\newblock PhD thesis, Universit\'e Paris-Diderot (Paris 7), 2016.

\bibitem{Gunaydin-vdD}
A.~G\"unaydin and L.~van~den Dries.
\newblock The fields of real and complex numbers with a small multiplicative
  group.
\newblock {\em Proceedings of London Mathematical Society}, 93(1):43--81, 2006.

\bibitem{Haskell-Macpherson}
D.~Haskell and D.~Macpherson.
\newblock A version of $o$-minimality for the $p$-adics.
\newblock {\em The Journal of Symbolic Logic}, 62(4):1075--1092, 1997.

\bibitem{vdD-H-M}
D.~Haskell, D.~Macpherson, and L.~van~den Dries.
\newblock One-dimensional p-adic subanalytic sets.
\newblock {\em J. London Math. Soc.}, 59(1):1--20, 1999.

\bibitem{Hodges}
W.~Hodges.
\newblock {\em Model theory}.
\newblock Cambridge University Press, 1993.

\bibitem{Lang}
S.~Lang.
\newblock {\em Algebra}.
\newblock Springer, {T}hird edition, 2002.

\bibitem{Laurent}
M.~Laurent.
\newblock {\'E}quations diophantiennes exponentielles.
\newblock {\em Inventiones Mathematicae}, 78(2):299--327, 1984.

\bibitem{Mann}
H.~Mann.
\newblock On linear relations between roots of unity.
\newblock {\em Mathematika}, 12:107--117, 1965.

\bibitem{Mariaule2}
N.~Mariaule.
\newblock Effective model-completeness for $p$-adic analytic structures.
\newblock arXiv:1408.0610, 2014.

\bibitem{Mariaule}
N.~Mariaule.
\newblock The field of $p$-adic numbers with a predicate for the power of an
  integer.
\newblock {\em Journal of the Symbolic Logic}, 82(1):166--182, 2017.

\bibitem{Neumann}
B.H. Neumann.
\newblock Groups covered by permutable subsets.
\newblock {\em Journal of the {L}ondon {M}athematical {S}ociety}, 29:236--248,
  1954.

\bibitem{Prestel}
A.~Prestel and P.~Roquette.
\newblock {\em Formally $p$-adic {F}ields}.
\newblock Springer-Verlag, 1984.

\bibitem{Robinson}
A.~Robinson.
\newblock Solution of a problem of {T}arski.
\newblock {\em Fundamenta Mathematicae}, 47:179--204, 1959.

\bibitem{Simon}
P.~Simon.
\newblock {\em A guide to NIP theories}.
\newblock Lecture in Notes in Logic. Cambridge University Press, 2015.

\bibitem{vdDries5}
L.~van~den Dries.
\newblock The field of reals with a predicate for the powers of two.
\newblock {\em {M}anuscripta Mathematica}, 54(1-2):187--195, 1985.

\bibitem{vanderPoorten}
A.J. van~der Poorten and H.P. Schlickewei.
\newblock Additive relations in fields.
\newblock {\em J. Australian Math. Soc.}, 51(1):154--170, 1991.

\bibitem{Zilber}
B.~Zilber.
\newblock A note on the model theory of the complex field with roots of unity.
\newblock Available at \emph{https://people.maths.ox.ac.uk/zilber/}, 1990.

\end{thebibliography}

\noindent Nathana\"el Mariaule\\
Universit\'e de Mons, Belgium\\
\emph{E-mail address: Nathanael.MARIAULE@umons.ac.be}

\end{document}